\documentclass{amsart}
\usepackage{amssymb}

\usepackage[all]{xy}
\usepackage{eurosym}

\newcommand\Int{\mathbf{N}}

\newcommand\oG{\overline{G}}

\newcommand\oV{\overline{V}}

\newcommand\mrurl[1]{MR. \url{http://www.ams.org/mathscinet-getitem?mr=#1}}
\newcommand\aurl[1]{\url{file:///data/ms/archive/#1}}
\newcommand\durl[1]{doi. \url{http:///dx.doi.org/#1}}
\newcommand\purl[1]{\url{file:///data/ms/print/#1}}
\newcommand\surl[1]{\url{file:///data/ms/screen/#1}}
\newcommand\zurl[1]{Zbl. \url{http://www.emis.de/zmath-item?#1}}
\newcommand\Aurl[1]{Arxiv. \url{http://arxiv.org/abs/#1}}
\newcommand\datver[1]{\def\datverp%
 {\par\boxed{\boxed{\text{#1; Run: \today}}}}}

\newcommand\boxb[1]{\square_b}

\newcommand\ff{\operatorname{ff}}

\numberwithin{equation}{section}

\newcommand\paperbody%
        {}

\newtheorem{lemma}{Lemma}
\newtheorem{proposition}{Proposition}
\newtheorem{corollary}{Corollary}
\newtheorem*{corollary*}{Corollary}
\newtheorem{theorem}{Theorem}
\newtheorem{non-theorem}{Non-Theorem}

\newtheorem*{conjecture*}{Conjecture}

\theoremstyle{remark}
\newtheorem{definition}{Definition}
\newtheorem{remark}{Remark}

\newtheorem{exam}{Example}

\newcommand\bV{\mathcal{V}_{\operatorname{b}}}

\newcommand\bo{\operatorname{b}}

\newcommand\scl{\operatorname{sl}}

\newcommand\ad{\operatorname{ad}}

\newcommand\bT{{}^{\bo}T}

\newcommand\cFTs{{}^{\Phi}\overline{T}\kern-1pt{}^*}

\newcommand\tY{\tilde{Y}}

\newcommand\FC{\mathfrak{C}}
\newcommand\FG{\mathfrak{G}}
\newcommand\FP{\mathfrak{P}}
\newcommand\FQ{\mathfrak{Q}}
\newcommand\FR{\mathfrak{R}}

\newcommand\cB{\mathcal{B}}

\newcommand\cD{\mathcal{D}}

\newcommand\cE{\mathcal{E}}
\newcommand\cF{\mathcal F}

\newcommand\cM{\mathcal M}
\newcommand\cS{\mathcal S}
\newcommand\cSp{{\mathcal S}'}

\newcommand\cG{\mathcal{G}}
\newcommand\cI{\mathcal{I}}

\newcommand\cP{\mathcal{P}}

\newcommand\cQ{\mathcal{Q}}

\newcommand\cR{\mathcal{R}}

\newcommand\cW{\mathcal{W}}

\newcommand\dcF{{}^2\kern-1.5pt\mathcal{F}}
\newcommand\DcA{{}^2\kern-3pt\mathcal{A}}
\newcommand\DcF{{}^2\kern-1.5pt\mathcal{F}}
\newcommand\DcL{{}^2\kern-1.5pt\mathcal{L}}

\newcommand\bbB{\mathbb B}
\newcommand\bbC{\mathbb C}

\newcommand\bbR{\mathbb R}
\newcommand\bbS{\mathbb S}

\newcommand\CI{{\mathcal{C}}^{\infty}}

\newcommand\Diag{\operatorname{Diag}}

\newcommand\Dff{\operatorname{Dff}}

\newcommand\cFNs{{}^{\Phi}\overline N\kern-1pt{}^*}

\newcommand\Id{\operatorname{Id}}

\newcommand\ci{${\mathcal{C}}^\infty$}

\newcommand\ha{\frac{1}{2}}

\newcommand\nul{\operatorname{null}}

\newcommand\oqu{\frac{1}{4}}

\newcommand\pa{\partial}

\newcommand\inn{\operatorname{int}}

\newcommand\Mand{\text{ and }}

\newcommand\Mat{\text{ at }}

\newcommand\Mif{\text{ if }}

\newcommand\Min{\text{ in }}

\newcommand\Mon{\text{ on }}
\newcommand\Mor{\text{ or }}

\newcommand\Msatisfy{\text{ satisfy }}

\newcommand\Mto{\text{ to }}
\newcommand\Mthen{\text{ then }}
\newcommand\Mwhere{\text{ where }}

\newcommand\ifitem[2]{\def\test{#2}\ifx\test\@empty\else\item #1#2\fi}

\newcommand{\dar}[1]{\ar@<2pt>[#1]\ar@<-2pt>[#1]}
\newcommand{\tar}[1]{\ar@<3pt>[#1]\ar@<-3pt>[#1]\ar[#1]}
\newcommand{\dscl}{\operatorname{2scl}}
\newcommand{\bof}{\operatorname{\bo\kern-1.5pt\phi}}

\datver{0.1A; Revised: 26-1-2003}
\begin{document}
\keywords{simplicial space, generalized product, b-fibration, p-embedding,
  pseudodifferential operator, semiclassical calculus, manifold with corners}
\subjclass{}
\title[GenProd]
{Generalized products and Semiclassical Quantization}
\author{Richard Melrose}
\address{Department of Mathematics, Massachusetts Institute of Technology}
\email{rbm@math.mit.edu}

\begin{abstract} The notion of a generalized product, refining that of a
  (symmetric and smooth) simplicial space is introduced and shown to imply
  the existence of an algebra of pseudodifferential operators. This
  encompasses many constructions of such algebras on manifolds with
  corners. The main examples discussed in detail here are related to the
  semiclassical (and adiabatic) calculus as used in the approach to a
  twisted form of the Atiyah-Singer index theorem in work with Is Singer
  and Mathai Varghese. Other examples will be discussed elsewhere.
\end{abstract}

\maketitle

\section*{Introduction}

One aim of this paper is to give a full discussion of the (well-known)
semiclassical calculus (see for example the book of Zworski
\cite{MR2952218}), as used in the context of index theory in
\cite{IndBunGer,MR2140985,MR2674880}. This includes `extensions' to the
semiclassical calculus such as the double semiclassical calculus for an
iterated fibration. The construction of these operator algebras follows a
pattern that has emerged over a long period with calculi such as the
b-calculus, the scattering calculus, the zero calculus and their many
variants. So rather than repeat the (non-automatic) form of these the
notion of a \emph{generalized product} is introduced and shown to determine
an algebra of operators. This conjecturally includes essentially all the
known algebras obtained this way, although that is not discussed here,
except for the b-calculus itself.

A generalized product is modeled on the fibre products of a fibration. The
notion is a refinement of that of a symmetric simplicial space; for the
topological concept see for instance \cite{MR1361886}. Abstractly we
consider a contravariant functor from the integers (or finite sets) to the
category of manifolds, here we are mainly interested in the target category
being that of the compact manifolds with corners. The morphisms of the
integers are all the maps from $J(m)$ to $J(n)$ where $J(m)=\{1,\dots,m\}.$
There are important additional constraints on the functor. Namely the
dimensions of the manifolds associated to the integers, written generically
$M[k],$ are required to be of the form $\dim M[k]=\mu+\kappa (k-1)$ for
integers $\mu$ and $\kappa.$ The maps associated to injections are required
to be b-fibrations and those associated to surjections to be
p-embeddings. These are special morphisms in the category of manifolds with
corners where the morphism are b-maps; the notion of a b-fibration is close
to, but crucially weaker than, that of a fibration. The precise definition
is given in \S\ref{GP} which also sets out the stronger notion of a
stretched product.

The multidiagonals, just called diagonals below, in a generalized product,
correspond to the surjective maps for $J(m)$ to $J(n)$ ($n<m$) and are of
particular interest as the carriers of the singularities of the kernels of
pseudodifferential operators. In Theorem~\ref{GP.217} it is shown that the
intersection properties of these diagonals are essentially the same as the
diagonals in a fibre product. In particular the fact that \emph{the}
diagonal, $D\subset M[2],$ is a p-embedded manifold allows the definition
of pseudodifferential operators in terms of kernels with conormal structure
at the diagonal, to be taken over directly. These operators are naturally
the `microlocalization' of a Lie algebroid over $M[1]$ although this is by
no means a general Lie algebroid; see the review article by Eckhard
Meinrenken \cite{2401.03034}. The precise global structure of the $M[k]$
remains to be elucidated.

Very special examples of generalized products arise from a compact Lie
group with $G[1]$ a point and the algebra resulting being the convolution
algebra. We conjecture that any real Lie group with connected centre arises
as the interior of a generalized product, i.e.\ has a `simplicial
compactification'. Two (essentially familiar) examples of this arise as the
compactifcations of the multiplicative group of the positive reals and the
translation group given by a vector space. These two cases are closely
connected to the b-calculus in the former case and the semiclassical (or
scattering) calculus in the latter.

Some results here, especially in \S\ref{BU} are derived from an
unfinished project with Michael Singer on the scattering stretched product;
I also thank him for comments on the manuscript. Helpful conversation with
Chris Kottke regarding generalized products are also happily acknowledged.

The notion of a generalized product is introduced in \S\ref{GP}. The
category of manifolds with corners is briefly described in \S\ref{MWC} and
the basic properties of blow-up of p-submanifolds is recalled in
\S\ref{BU}. This is applied inductively to give the resolution of a pic
(intersection-closed p-clean collection of submanifolds) in size order in
\S\ref{Res-size}. In \S\ref{Res-int} it is shown that the same resolution
arises from blowing up in an intersection order. The properties of
diagonals in a generalized product are discussed in \S\ref{MD} and the Lie
algebroid associated to a generalized product is derived in
\S\ref{LA}. This Lie algebroid is quantized to an alebra of
pseudodifferential operators in \S\ref{Ops}. In \S\ref{SGP} a semiclassical
extension on any generalized product is introduced and used to show the
existence of an adiabatic stretched product associated to fibration. The
fundamental b-streched product can be found in \S\ref{BSP}. The double
semiclassical product, corresponding to an iterated fibration, in
\S\ref{DSST} arises in the semiclassical proof of the Atiyah-Singer index
theorem in work with Singer and Varghese.

\paperbody
\section{Generalized and stretched products}\label{GP}

An important example of a simplicial space arises from a smooth fibre
bundle 
\begin{equation}
\phi:M\longrightarrow Y.
\label{GP.120}\end{equation}
In case $M$ is a compact manifold without boundary, a submersion
induces such a fibre bundle. In the case of non-compact manifolds and also in
the category of compact manifolds with corners, recalled below, it is
necessary to require more in order that $\phi$ have local trivializations
over a suitable cover of $Y.$ Given this, the fibre products form a
simplicial space
\begin{equation}
\xymatrix{
M&\dar{l} M^{[2]}_\phi&\tar{l} M^{[3]}_{\phi}\cdots.
}
\label{2.7.2023.1}\end{equation}
Here the maps are each a projection off one of the factors.

Going somewhat beyond the standard notion of a simplicial space, observe
that these maps and spaces correspond to part of a contravariant functor
from the category with objects the integers $\Int,$ identified as the
collection of sets $J(n)=\{1,\dots,n\},$ and with morphisms all the maps
from $J(n)$ to $J(m),$ (rather than just the order-preserving maps) to the
category of manifolds. Thus if
\begin{equation}%
\begin{gathered}
I:\{1,\dots,n\}\longrightarrow \{1,\dots,m\}\Mthen\\
S_I:M^{[m]}_{\phi}\longrightarrow
M^{[n]}_{\phi},\ S_I(p_1,\dots,p_m)=(p_{I(1)},\dots,p_{I(n)}).
\end{gathered}
\label{SCL.47}\end{equation}
The maps $S_I$ corresponding to injective maps $I$ are fibrations, namely
up to relabelling the various projections off factors in
\eqref{2.7.2023.1}. The maps corresponding to surjective maps $I$ are the
embeddings of `multidiagonals'. However in the case of fibres which have
boundary these do not satisfy the stronger condition we will require of
being p-embeddings, that near boundary points there is a common local
product decomposition of manifold and submanifold.

For instance for $n=2$ and $m=1$ so $I(1)=I(2)=1,$  
\begin{equation}
S_I:M\ni p\longrightarrow (p,p)\in M^{[2]}_{\phi}
\label{SCL.48}\end{equation}
is the embedding of the diagonal in the fibre product. If the fibre has a
boundary then the diagonal passes through the corner and does so in a
manner which precludes a common product decomposition of manifold and
submanifold. This is a particular issue for the kernels of
pseudodifferential operators.


So, for the definition, note that in the category of (usually compact)
manifolds with corners the morphisms are b-maps and there are special
surjective maps, the b-fibrations, and special injective maps, the
p-embeddings. This is all discussed in \S\ref{MWC} below.

\begin{definition}\label{SCL.43} A \emph{generalized product} of manifolds
  (with corners) is a contravariant functor from the category $\Int,$ with
  morphisms \emph{all} the maps from $J(n)$ to $J(m),$ to the category of
  manifolds with corners, with objects denoted $M[k],$ $k\in\Int,$ $\dim
  M[k]=\mu+(k-1)\kappa,$ morphisms the interior b-maps and under the functor
  injective maps are taken to b-fibrations and surjective maps to
  p-embeddings; all manifolds (including the fibres of b-fibrations) are
  assumed to be connected.
\end{definition}

\noindent We call the images $S_I:M[m]\longrightarrow M[n]$ of the maps
$I:J(m)\longrightarrow J(n)$ the \emph{structure maps}.

The main interest lies in the compact case; the assumption of connectedness
is made to simplify the discussion below, it is not really crucial.

\begin{definition}\label{GP.234} A generalized product is said to be
  \emph{compact} if all the spaces $M[k]$ are compact (and connected)
  and the structural b-fibrations have connected fibres.
\end{definition}

The permutations $\sigma:J(n)\longrightarrow J(n)$ for each $n$ must induce
maps which are both b-fibrations and p-embeddings and hence are
diffeomorphisms, the structural diffeomorphisms. Thus the corresponding
$S_\sigma$ acting on $M[n]$ induce an action of the permutation group.

\begin{lemma}\label{GP.127}
A generalized product is symmetric.
\end{lemma}

In any generalized product the minimal, or $k$-fold, diagonal in $M[k]$ is
the p-submanifold $D=D_k:M[1]\longrightarrow M[k]$ corresponding to the
unique map $J(k)\longrightarrow J(1).$ Consider one of the $k$ b-fibrations
$\Pi:M[k]\longrightarrow M[1]$ arising from a map $J(1)\longrightarrow
J(k).$ Since the composite $J(1)\longrightarrow J(k)\longrightarrow
J(1)$ is the identity $\Pi$ restricts to the identity map on the image of
$D_k.$ If $H$ is a boundary hypersurface of $M[1]$ it follows that the
inverse image $\Pi^{-1}(H)$ is a union of boundary hypersurfaces of $M[k]$
exactly one of which meets $D_k$ (since it is an interior p-submanifold)
and does so in $H.$ We denote this boundary hypersurface as $H[k].$

\begin{proposition}\label{GP.136} If $M[k]$ is a generalized product and
  $H$ is a boundary hypersurface of $M[1]$ the boundary hypersurfaces
  $H[k]$ have an induced structure as a generalized product.
\end{proposition}

\begin{proof} Clearly $\dim H(k)=(\mu -1)+\kappa (k-1).$ The composite of
  any map $J(n)\longrightarrow J(m)$ with $J(m)\longrightarrow J(1)$ is the
  unique map $J(n)\longrightarrow J(1).$ Taking $m=n=k$ it follows that any
  of the structure diffeomorphisms map the maximal diagonal $D_k$ into
  itself as the identity. As diffeomorphisms they map boundary
  hypersurfaces to boundary hypersurfaces so map $H[k]$ into itself giving
  the permutation action on $H[k].$

Similarly for any of the structure maps $S_I$ for $I:J(n)\longrightarrow
J(m)$ the composite with $J(m)\longrightarrow J(1)$ is the map
$J(n)\longrightarrow J(1)$ so $S_ID_m=D_n.$ Thus $S_I$ maps the image
of $D_m$ to the image of $D_n$ as the identity. Again it maps boundary
hypersurfaces into boundary faces so must map $H[n]$ to $H[m].$ So each of
the structure maps takes $H[m]$ into $H[n]$ and the functorial properties
follow.
\end{proof}

Since the spaces $M[1],$ $M[2]$ and $M[3]$ are particularly important we
use special notation for some of the structure maps. Thus 
\begin{equation}
\begin{gathered}
\Pi_L,\ \Pi_R:M[2]\longrightarrow M[1]\Mand\\
\Pi_S,\ \Pi_C,\ \Pi_F:M[3]\longrightarrow M[2]
\end{gathered}
\label{GP.317}\end{equation}
are the b-fibrations which are the images of the maps $L,$
$R:J(1)\longrightarrow J(2),$ $L(1)=1,$ $R(1)=2$ and $S,$ $C,$
$F:J(2)\longrightarrow J(3)$ where $S(1)=1,$ $S(2)=2,$ $C(1)=1,$ $C(2)=3$
and $F(1)=2,$ $F(2)=3.$ Similarly the partial diagonals are the p-embeddings
\begin{multline}
D_{i,j}:M[2]\longrightarrow M[3],\ i\not=j,\\
\text{corresponding to the
  surjective maps taking }\{i,j\}\Mto1
\label{GP.318}\end{multline}
and $R_{i,j}$ for $i\not=j$ correspond to the reflections exchanging $i$
and $j.$ The single reflection on $M[2]$ is denoted $R.$

By assumption all the structure maps of a generalized product are interior
b-maps and so have the property that the inverse image of the boundary of
the range space is contained in the boundary of the domain. It follows that
they map the interior into the interior in all cases.

\begin{lemma}\label{GP.275} The interiors $\inn(M[k])$ have an induced
  structure as a generalized product.
\end{lemma}

So generalized products can be thought of as compactifications of their
interiors. Restricted to the interiors consider the maps $\Pi_k=S_I,$
$I:J(k-1)\longrightarrow J(k)$ being the strictly increasing map $I(j)=j$
and $D_{1,2}=S_L$ $L:J(k)\longrightarrow J(k-1)$ being, for general $k\ge3,$ the surjective
increasing map with $L(1)=L(2)=1.$ Then
\begin{equation}
\begin{gathered}
\Pi_k:\inn M[k]\longrightarrow \inn M[k-1]\text{ is a fibration and}\\
D_{1,2}:\inn M[k-1]\longrightarrow \inn M[k]\text{ is an embedding for each }k.
\end{gathered}
\label{GP.79}\end{equation}

Conversely
\begin{lemma}\label{GP.81} If $M[k]$ are compact manifolds with
  corners with interiors which form a generalized product then to show that
  the $M[k]$ form a generalized product, and hence a compactification, it
  suffices to show that the structural diffeomorphisms
  extend to diffeomorphisms of each $M[k],$ the fibrations
  $\Pi_k:M[k]\longrightarrow M[k-1]$ extend to b-fibrations and the embeddings
  $D_{1,2}:M[k-1]\longrightarrow M[k]$ extend to p-embeddings.
\end{lemma}

\begin{proof} For any map $I:J(n)\longrightarrow J(m)$ suppose the range
  has $k$ elements, then it can be identified with $J(k)$ and $I$ is the
  composite of a surjection to $J(k)$ and an injection from $J(k).$ So the functorial properties
  imply that it suffices to show the existence of all the $S_I$ for
  surjections and injections and the special case of bijections, the others
  can then be obtained by composition. So, assuming the diffeomorphisms
  giving the action of the permutation group on the interior extend to be
  smooth b-maps, symmetry follows for the $M[k].$

After composition with permutations on both sides, an injection is reduced
to a map $I:J(n)\longrightarrow J(m)$ where $I(j)=j.$ These in turn are
products of the maps $I:J(n)\longrightarrow J(n+1),$ $I(j)=j.$ Similarly
the surjections can be decomposed into products of permutations and the
maps corresponding to the $D_{1,2}.$
\end{proof}

As is easily seen
\begin{lemma}\label{GP.277} The product $M[k]\times N[k]$ of two
  generalized products is a generalized product.
\end{lemma}

\noindent However the product may not be the most natural compactification
of the products of the interiors.

\begin{exam}\label{GP.294} The basic example of a generalized product
  arises, as noted above, as the fibre products $M[k;\phi]=M^{[k]}$ of a
  fibration of compact manifolds \eqref{GP.120}. This `resolves' (and in
  consequence `quantizes') the Lie algebra of fibre vector fields on $M.$
  These form the Lie algebroid over $M=M[1]$ formed by the sections of the
  null bundle, $\nul(\phi_*),$ over $M=M[1].$

For later reference we give a `simplicial' construction of this Lie algebra
structure using $M[2]$ and $M[3].$ The two projections $\pi_L,$
$\pi_R:M[2]\longrightarrow M[1]$ are fibrations so the null bundles of
$(\pi_L)_*$ and $(\pi_R)_*$ are subbundles of $TM[2],$ namely the right and
left fibre tangents. We identify the bundle $\nul(\phi_*)$ over $M[1]$
more abstractly as a vector bundle over $M[1]$ 
\begin{equation}
E=D^*\nul((\pi_R)_*).
\label{GP.295}\end{equation}
Clearly $(\pi_L)_*$ identifes $\nul((\pi_R)_*)$ over the (image of) the
diagonal with $\nul(\phi_*).$ 

Consider the following indirect, but more simplicial construction of the
Lie algebra structure on $\CI(M[1];E).$ Starting from a section $V\in
\CI(M[1];E)$ we first identify it as a section of $\nul((\pi)_R)_*)$
over $D.$
\begin{enumerate}
\item The projection $\Pi_S$ maps $D_{1,2}$ as a fibration onto $D.$ The null
  space of the differential $\nul((\Pi_F)_*)$ over $D_{1,2}$ is mapped by
  $(\Pi_S)_*$ isomorphically onto $\nul((\pi_R)_*$ at each point of
  $D_{1,2}$ so $V$ lifts under $(\Pi_S)_*$ to a smooth section
  $W\in\CI(D_{1,2};\nul((\pi_F)_*)).$
\item The fibration $\Pi_C$ restricts to a diffeomorphism
  $D_{1,2}\longrightarrow M[2]$ and $(\Pi_C)_*$ maps
  $\nul((\pi_F)_*))$ isomorphically onto $\nul((\Pi)_R)_*)$ over $M[2].$
  This defines $\tilde V=(\Pi_C)_*W\in\CI(M[2];\nul((\Pi)_R)_*).$ The
  restriction $\tilde V\big|_D=V.$
\end{enumerate}

These claims are easily verified in a local trivialization of $\phi$ over
$M$ with base coordinates $y$ and fibre coordinates written $(z,z',z'')$ in
$M[3]$ and $(z',z'')$ in $M[2]$ induced by fibre coordinates $z$ in $M.$
Then $V=\sum\limits_{i}v_i(y,z')\cdot\pa_{z'}$ as a section of
$\nul((\pi_R)_*$ over $D=\{z'=z''\}.$ The pull-back to $D_{1,2}$ as a
section of $\nul((\pi_F)_*)$ is
\begin{equation*}
W=\sum\limits_{i}v_i(y,z)\cdot\pa_{z}. 
\label{GP.296}\end{equation*}
This pushes forward under $(\pi_C)_*$ to the vector field 
\begin{equation*}
\tilde V=\sum\limits_{i}v_i(y,z')\cdot\pa_{z'}
\label{GP.297}\end{equation*}
on $M[2]$ which extends the original vector field $V$ to be fibre constant
on the `right' fibres of $M[2].$ These form a Lie subalgebra of the
sections of $\nul((\pi_R)_*)$ over $M[2]$ (which can depend on the `right'
variable $z'').$ This defines a Lie algebra structure on $\CI(M[1];E)$ which
is identified with with the obvious structure of fibre vector fields on
$M.$

Whilst trivial in this case, this indirect construction is the basis of the
identification of the Lie algebroid structure in the wider context of a
generalized product in \S\ref{LA} below.
\end{exam}

\begin{exam}\label{GP.156} Let $G$ be a connected Lie group. Then the products
  $G^{k}$ for $k=1,2,\dots,$ form a generalized product (not using the
  group structure) which is equivariant under the right action of $G$ via
\begin{equation*}
G^k\times G\ni((g_1,\dots,g_k),g)\longrightarrow
(g_1g^{-1},\dots,g_kg^{-1})\in G^k,
\label{GP.137}\end{equation*}
so $S_I(\gamma g^{-1})=S_I(\gamma )\cdot g^{-1}$ for all $I.$ The quotients  
\begin{equation}
G[k]=G^k/G
\label{GP.138}\end{equation}
therefore form a generalized product.

For instance the three symmetries on $G^2$ corresponding to the involutions
exchanging the first two, first and third and second and third elements are 
\begin{multline}
R_{1,2}(a,b)=(b,a),\ R_{1,3}(a,b)=(a^{-1},ba^{-1}),\
R_{2,3}(a,b)=(ab^{-1},b^{-1})\\
\Mon G[3]=G^2.
\label{GP.140}\end{multline}
The reflection, $R,$ on $G=G[2]$ is $g\longmapsto g^{-1},$ the three projections
from $G[3]=G^2$ to $G[2]$ are
\begin{equation*}
\Pi_{F}(a,b)=b,\ \Pi_C(a,b)=a,\ \Pi_S(a,b)=ab^{-1},
\label{GP.141}\end{equation*}
the three diagonals, in $G[3],$ which are diffeomorphic to $G,$ are 
\begin{equation*}
D_{1,2}=\{(g,g)\},\ D_{1,3}=\{(\Id,g)\},\ D_{2,3}=\{(g,\Id)\}
\label{GP.142}\end{equation*}
and the triple diagonal is the point $D=(\Id,\Id).$

The projections $\Pi_L=\Pi_R$ both map to a point so
$\nul((\Pi)_R)_*)=T_{\Id}G$ is a `bundle' over $G[1].$ Consider the
analogues of the operations carried out in the case of a fibre bundle above.
\begin{enumerate}
\item The null space of $(\Pi_F)_*$ at a point $(g,g)\in D_{1,2}$ is
$T_gG\times \{0\}.$ Under $\Pi_S,$ $(e^{tv}g,g)$ is mapped to $e^{tv}$
  so $W=(R_v,0)$ on $D_{1,2}.$
\item Thus the image $(\Pi_C)_*W$ is $\tilde v=R_v,$ the right-invariant
  vector field on $G$ evaluating to $v$ at $\Id.$ 
\end{enumerate}

Thus the simplicial construction, formally the same as that for a fibration,
yields the (right) Lie algebra structure on $E=T_{\Id}G.$
\end{exam}

\begin{definition}\label{GP.144} A \emph{simplicial compactification} of a
  Lie group, $G,$ is a generalized product $\oG[*],$ in the category of
  compact manifolds with corners, reducing to \eqref{GP.138} in the
  interior.
\end{definition}

\noindent Of course for a compact Lie group \eqref{GP.138} is a
simplicial compactification.

\begin{conjecture*}\label{GP.143} Any real reductive Lie group with connected
  centre has a simplicial compactification.
\end{conjecture*}

It may be that this can be seen using results such as in \cite{1910.02811}
and particularly the thesis of Malte Behr \cite{Behr-thesis}. Two (commutative)
examples of simplicial compactification arise below. For a vector space, as
an abelian Lie group, a simplicial compactification arises from the
construction of the semiclassical stretched product below, through
restriction to the (one) boundary hypersurface as in
Theorem~\ref{GP.136}. For the multiplicative group of the positive real
numbers a simplicial compactification is produced in \S\ref{BSP}, it is not
the same as the compactification of the additive group on $\bbR,$ even
though the two groups are diffeomorphic.

There is frequently a larger group of diffeomorphisms preserving the
structure maps of a generalized product beyond the permutation action. If a
group $K$ acts in the sense that there is a homomorphism $K\longrightarrow
\operatorname{Dif}(M[k])$ into the diffeomorphism group for all $k$ and
these are intertwined by the structural maps then one can form bundles with
fibres modelled on the generalized product and cocycle in $K.$

One special case of the notion of generalized product is particularly
common, that of a `stretched product'. For this we require that the
interiors be the fibre products of a fibration. In typical cases these
interiors are actually the interiors of a fibre bundle of compact manifolds
with corners, so that the stretched products typically arise by
`resolution' but we do not demand this explicitly.

\begin{definition}\label{GP.135} A \emph{stretched product} is a
  generalized product $M[k]$ in the category of compact manifolds with
  corners such that there is a fibration $\phi:\inn(M[1])\longrightarrow
  Y'$ and the interiors of the $M[k]$ are identified with their structural
  maps with the fibre products of $\phi.$
\end{definition}

It turns out that there may be different generalized products compactifying
a given open generalized product, and in particular a particular open fibre
product. So identifiers are used to distinguish them and we write $M[k;T]$
for the generalized or stretched product with identifier `$T$'. As already
suggested, in typical exampes of a stretched product the interior
fibration $\phi$ extends to a fibration of compact manifolds with corners
$\phi:M[1;T]\longrightarrow Y$ and the notation is intended to suggests
that $M[k;T]$ is defined through a resolution process, such as iterated
blow-up, from the fibre products of the $M=M[1;T].$ In the examples of
stretched products below the $M[k;T]$ are constructed by blow-up of
boundary p-submanifolds of these $M^{[k]}_{\phi}.$ This leads directly to
symmetry and then the maps $\Pi_k^T$ and $D_{1,2}^T$ in Lemma~\ref{GP.81}
are shown to exist through commutation and lifting properties of blow-up.

See the work of Kottke and Rochon, \cite{MR4624070}, for related
ideas of the construction of fibre products.

Proposition~\ref{GP.136} applies to stretched products but (as seen in the
adiabatic example below) the resulting generalized product for a boundary
hypersurface is not a stretched product. Whilst in principle `simpler' than
the stretched product from which they come these boundary generalized
products typically have a factor related to the group example above,
i.e.\ the associated Lie algebroid (discussed in \S\ref{LA} below) has a
non-trivial null space to its anchor map.

\section{Manifolds with corners}\label{MWC}
We  briefly recall the concepts involved in  Definition~\ref{SCL.43} and
used freely in the remainder of this paper.

First, to give a direct (if not particularly natural) definition of
manifolds with corners, consider a compact manifold, $\hat M,$ without
boundary and a \emph{hypersurface involution}. This is a smooth involution,
$R,$ of the manifold which has fixed points and at each of them the
differential of $R$ has exactly one negative eigenvalue. It follows that
\begin{equation*}
\hat H_R=\{m\in\hat M;Rm=m\}
\label{SCL.44}\end{equation*}
is a closed embedded separating hypersurface (possibly with several
components). In particular there exists $\xi\in\CI(\hat M)$ such that
$R^*\xi=-\xi$ and $\hat H=\{\xi=0\}$ with $d\xi\not=0$ at $\hat H.$

Suppose $\hat M$ is equipped with $N$ such commuting involutions, $R_j,$
$j=1,\dots,N$ and corresponding $\xi_j$ such that for any subset
$I\subset\{1,\dots,N\}$ and point $m$ at which $\xi_j(m)=0$ for all $j\in I$
the $d\xi_j(m)$ are linearly independent. Then the space
\begin{equation}
M=\{m\in\hat M;\xi_j\ge0\ \forall\ j\},\ \CI(M)=\CI(\hat M)\big|_{M}
\label{SCL.45}\end{equation}
is a compact manifold with corners. Provided it is required as part of the
definition that the boundary hypersurfaces of a manifold with corners be
embedded (as we do), this is equivalent to more standard definitions.

Taking this as the definition, the boundary hypersurfaces of $M$ are the
components of the fixed hypersurfaces $\hat H_i\cap M;$ the set of them
denoted $\cM_1(M).$ Each $H_j\in\cM_1(M)$ has a defining function $0\le
x_j\in\CI(M)$ such that $\CI(M)x_j$ is the ideal of functions vanishing on
$H_j.$ These $x_j$ are boundary defining functions of $M.$ Note that, since
the $\hat H_i\cap M$ may have several components, they are in general
disjoint unions of boundary hypersurfaces and the $\xi_i$ restricted to $M$
become `collective' boundary defining functions, i.e.\ products of the
corresponding $x_j.$ The dimension, $n,$ of $M$ is the dimension of $\hat
M.$ At any point of $M$ there are `adapted' coordinates consisting of
$k\ge0$ of the $x_j$ which vanish there and $n-k$ locally smooth functions
$y_l$ together giving a diffeomorphism of an open neighbourhood of the
point to $[0,\epsilon )^k\times (-\epsilon ,\epsilon )^{n-k}$ for some
  $\epsilon >0.$ We call these simply `coordinates' on $M.$

The boundary faces of codimension $k,$ forming $\cM_k(M),$ are the
components of non-trivial intersections of $k$ boundary hypersurfaces.

A map $F:M_1\longrightarrow M_2$ between compact manifolds with corners is
smooth if $F^*:\CI(M_2)\longrightarrow \CI(M_1).$ Writing complete sets of
boundary defining functions $x_j'$ for $M_2$ and $x_k$ for $M_1,$ a smooth
map is a \emph{b-map} if for each $j$ one of three possibilities holds
\begin{equation}
F^*x_j'\equiv0,\ F^*x_j'=\alpha _j\prod\limits_k x_k^{a(j,k)},\ \alpha
_j>0\Mor F^*x_j'>0.
\label{SCL.46}\end{equation}
Here the $a(j,k)$ are necessarily non-negative integers but may vanish and
the final case is the special form of the second when all the powers
vanish. If the first case does not occur, so the image of $F$ does
\emph{not} lie in a boundary hypersurface, $F$ is called an interior
b-map. Clearly the composite of b-maps is a b-map and the compact manifolds
with corners form a category with the b-maps as the morphisms (or with
interior b-maps as morphisms).

The smooth vector fields on $M$ which are tangent to all boundaries form a
Lie algebra $\bV(M)$ which is a module over $\CI(M)$ and locally has a
basis, i.e.\ $\bV(M)=\CI(M;\bT M)$ for a smooth vector bundle over $M$ with
a natural bundle map $\bT M\longrightarrow TM$ of corank $k$ at a boundary
point of codimension $k.$ In local coordinates $\bT M$ has the local basis
$x_j\pa_{x_j},$ $\pa_{y_l}.$ Thus $\bV(M)$ is a \emph{Lie algebroid}. An
interior b-map has a b-differential $F_*:\bT_mM_1\longrightarrow
\bT_{F(m)}M_2$ consistent with the usual differential. For a boundary b-map
one must replace $\bT M_2$ by $\bT B$ where $B$ is the smallest boundary
face containing the image of $F.$ 

To complete the explanation of the Definition~\ref{SCL.43} we need to
define p-embeddings and b-fibrations which are the analogues of embeddings
and fibrations in this category.

A \emph{p-submanifold} of $M$ is a closed (and generally assumed connected
here) subset everywhere locally of the form $x_i=0,$ $i=1,\dots,d',$
$y_l=0,\ l=1,\dots,n-p'$ in terms of some coordinates. Thus it has a
\emph{p}roduct decomposition consistent with that of the manifold. It is
always a manifold with corners. A \emph{p-embedding} is a diffeomorphism of
a manifold with corners onto a p-submanifold. An interior p-submanifold is
one with no component contained in a boundary hypersurface, equivalently it
is locally defined by the vanishing of interior coordinates.

A \emph{b-fibration} $F:M_1\longrightarrow M_2$ between compact manifolds
with corners is a surjective b-map with the property that near each point
$m\in M_1$ there are local coordinates $x,y$ and also local coordinates
$x',y'$ near $F(p)\in M_2$ such that the $F^*y'_{l'}$ are amongst the $y_l$
and
\begin{equation}
F^*x'_{j'}=\prod_{j\in J(j')} x_j^{\alpha (j',j)}\text{ where }J(j'_1)\cap
J(j'_2)=\emptyset \Mif j'_1\not=j'_2.
\label{SCL.51}\end{equation}
Thus each local defining function for the boundary of $M_1$ near $p$ can
occur in at most one of the product decompositions. The only case which
arises below is of a \emph{simple b-fibration} where the $\alpha (j,j')$
are either $0$ or $1;$ we therefore subsume this in the notion of a
b-fibration. As for a fibration (which corresponds to the case that at
most one $\alpha (j',j)=1$ for each $j')$ surjectivity is automatic if
$M_2$ is connected.

\begin{proposition}\label{GP.128} The inverse image of an interior
  p-submanifold under a b-fibration is an interior p-submanifold.
\end{proposition}

\begin{proof} This follows from the local structure of manifold and the map.
\end{proof}

\noindent The inverse image of a boundary
  p-submanifold under a b-fibration is a union of boundary p-submanifolds.

\section{Blow up}\label{BU}

We also recall the notion of blow-up, which is well-defined for any
p-submanifold $X\subset M$ (recall that these are by definition closed) of
a manifold with corners. The construction of $M$ blown up along the centre
$X$ is closely related to the existence of a collar neighbourhood, or
normal fibration, for $X$ in $M.$ The result is a manifold with corners
(functorially associated to the pair $(M,X))$ with associated blow-down
(interior b-)map
\begin{equation}
\beta :[M;X]\longrightarrow M.
\label{SCL.52}\end{equation}
We will usually consider the case where the centre of blow-up, $X,$ is a
boundary p-submanifold. As a set 
\begin{equation}
[M;X]=S^+X\sqcup (M\setminus X),
\label{SCL.53}\end{equation}
the blown-up space, is the disjoint union of the complement of $X$ in $M$
and the inward-pointing part of the spherical normal bundle to $X.$ Thus if
$p\in X$ and $x_j$ are a maximal set of $k$ local boundary defining
functions for $M$ near $p$ which vanish identically on $X$ (so $X$ is
contained in a boundary face of codimension $k)$ then
\begin{equation*}
S^+_pX=N^+_pX/T_pX,\ T^+_pX=\{v\in T_pM;vx_j(p)\ge0\}.
\label{SCL.54}\end{equation*}
These form a bundle of the $k$-positive parts of spheres,

\begin{equation*}
\bbS^{n,k}_+=\bbS^n\cap([0,\infty)^{k+1}\times \bbR^{n-k}),
\label{GP.348}\end{equation*}
over $X,$ where $n$ is the codimension of $X$ in $M.$ The well-defined
\ci\ structure on $[M;X]$ makes this bundle into a boundary hypersurface
(the `front face' of the blow-up, also the lift or proper transform of the
centre). The \ci\ functions near a point of the front face are those which
are smooth in polar coordinates in the normal variables to $X$ and in
tangential variables. The blow-down map is the identity away from the front
face and a fibration of the front face onto the centre.

\begin{lemma}\label{GP.278} A diffeomorphism, $F,$ of $M$ lifts to a
  diffeomorphism of $[M;X]$ to $[M;F(X)];$ so $F$ lifts to a diffeomorphism of
  $[M;X]$ if and only if it restricts to a diffeomorphism of $X.$
\end{lemma}

A basic result on blow-up is:

\begin{lemma}\label{GP.83} If $F_i\subset M$ are two p-submanifolds which
  intersect transversally (including being disjoint) or are nested in the
  sense that $F_1\subset F_2$ then
\begin{equation}
[[M;F_1];F_2]=[[M;F_2];F_1]
\label{GP.86}\end{equation}
are naturally diffeomorphic.
\end{lemma}

\begin{proof} Assume first that $F_1\subset F_2.$ Away from the preimage of
  $F_1$ both blow-ups reduce to the blow-up of $F_2\setminus F_1.$ Near
  $F_1$ both are bundles over $F_1$ with the blow-ups being in the
  fibres. Thus we are reduced to the case that $F_1$ is the origin in
  $(y,z)\in\bbR^p\times\bbR^q$ and $F_2=\{y=0\}.$ Then both blow-ups
  represent completions of $y\not=0,$ to a manifold with corners.

Blowing up $F_1$ first introduces $R=(|y|^2+|z|^2)^{\ha},$ $y/R,$ $z/R$ as
smooth functions generating the \ci\ structure. Then blowing $y/R=0$ the
functions $r=|y|/R$ and $y/|y|$ become smooth in $|y|<\epsilon R.$ Note
however that in $|y|>\ha\epsilon R,$ $R/|y|$ is smooth so in fact $y/|y|$
and $r$ are globally smooth and induce the \ci\ structure on the blow-up
with $r$ and $R$ as boundary defining functions.

Proceeding in the opposite order the gives first the polar coordinates
$s=|y|$ and $y/|y|,$ then blowing up $s=0=|z|$ introduces
$R=(s^2+|z|^2)^{\ha},$ $r=s/R$ and $z/R$ as smooth functions. This way the
manifold looks like the product of the sphere in $y$ and the half-space in
$(s,z)$ blown up at the origin, i.e.\ the product of a sphere, a
half-sphere and a half-line in $R.$ 

It follows that the functions on the complement of the larger submanifold
which extend to define the \ci\ structures on the two resolutions are the
same, so the two blow-ups are naturally diffeomorphic.
\end{proof}
This natural diffeomorphism shows that the blow-up of a sphere along a
coordinate subsphere is the product of a sphere and a half-sphere.

\begin{definition}\label{GP.292}  A finite collection of
p-submanifolds is \emph{p-clean} (generalizing the notion due to Bott)
if near each point of the union there are local coordinates $x_j,$ $y_l$ in
which each submanifold in the set, containing that point, is defined by a
subset of the $x_j$ and some linear functions of the $y_l.$
\end{definition}

For any two manifolds in such a p-clean collection there is, near each
point of intersection, a joint product decomposition with the manifold.

\begin{proposition}\label{GP.85} If $F_i\subset M,$ $i=1,2,3$ are a p-clean
triple satisfying 
\begin{equation}
F_3\subset F_1,\ F_2\setminus F_1\not=\emptyset,\ F_3\text{
  transversal to }F_2\cap F_1\text{ within }F_1
\label{GP.88}\end{equation}
then $F_2$ and $F_3$ lift to be transversal in $[M;F_1].$ 
\end{proposition}

\begin{proof} After the blow-up of $F_1,$ the lift of $F_2$ is transversal
  to the front face. So it suffices to show that the intersection of $F_2$
  with the front face is transversal to the lift of $F_1.$ The latter is
  the union of fibres of the front face, so this reduces to the
  transversality of the quotients by the blow up within $F_1$ which holds
  by hypothesis.
\end{proof}

\begin{proposition}\label{GP.117} If $\phi:M\longrightarrow N$ is a
  (simple) b-fibration of compact manifolds with corners and $C\subset M$ is a
  boundary p-submanfold not contained in a boundary face of codimension
  greater than one and such that $\phi:C\longrightarrow H$ is a b-fibration
  with image a boundary hypersurface of $N$ then the composite 
\begin{equation}
\beta\circ\phi:[M;C]\longrightarrow N\text{ is a b-fibration.}
\label{GP.118}\end{equation}
\end{proposition}

\begin{proof} The composite \eqref{GP.118} is certainly a b-map. Suppose
  $C\subset H'$ where $H'$ is the hypersurface containing $C$ -- by
  assumption it is then the minimal boundary face doing so, i.e.\ interior
  points of $C$ have codimension $1$ in $M.$ Consider the preimage in $C$
  of a general point $p\in H$ of codimension $q+1,$ $q\ge0,$ i.e.\ having
  codimension $q$ as a point in $H.$ Let $m\in C$ be in the preimage under
  $\phi$ of $p.$ Near $m$ there are local coordinates in $N$ near $p,$ $x,$
  $y$ and  $\xi$ $y,z$ in $M$ near $m$ such that $H=\{x_1=0\},$
  $C=\{\xi_1=0,\ z=0\}$ with $\phi^*x_1=\xi_1$ and the $\phi^*x_j$ $j>1$
  products (without common factors) of the $\xi_p,$ $p>1.$ Thus
  $R=(|z|^2+\xi_1^2)^{\ha}$ is a defining function for the front face of
  $[M;C]$ and $r=\xi_1/R$ is a defining function for the lift of $H'.$ Thus
  the lift of $x_1$ becomes $rR$ and the other lifts remain unchanged. It
  follows that $\phi\circ\beta$ is everywhere locally, and hence globally, a
  b-fibration. 
\end{proof}
\begin{proposition}\label{GP.352} If $\phi:M\longrightarrow N$ is a
  (simple) b-fibration of compact manifolds with corners and $Y\subset N$
  is a p-submanifold such that $\phi^{-1}(Y)\subset M$ is also a
  p-submanifold then $\phi$ lifts to a b-fibration 
\begin{equation}
\tilde\phi:[M,\phi^{-1}(Y)]\longrightarrow [N,Y].
\label{GP.353}\end{equation}
\end{proposition}

\begin{proof} In case $Y$ is an interior p-submanifold then $\phi^{-1}(Y)$
  is automatically a p-submanifold. Locally any near point
  $q\in\phi^{-1}(Y)$ and its image $p=\phi(q)$ there are local product
  neighbourhoods $B\times B'$ in $N$ and $B\times B''$ in $M$ where
  $\{b\}\times B'$ is a neighbourhood of $p\in Y,$ $\{b\}\times B''$ is a
  neighbourhood of $q\in \phi^{-1}(Y)$ and $\phi$ is a product map which is
  the identity on $B.$ From this the result follows.

In case $Y$ is a boundary p-submanifold, the assumption that $\phi^{-1}(Y)$
is a (connected) p-submanifold is non-trivial. If $Q$ is the smallest
boundary face of $N$ containing $Y$ it follows that $\phi^{-1}(Q)$ is a
boundary face of $M,$ rather than a union of such. In particular each of
the boundary hypersurfaces of $N$ containing $Y$ has preimage just one
boundary hypersurface of $M.$ Then there is a similar product decomposition
to the interior case and the result follows.
\end{proof}

\begin{proposition}\label{GP.331} If $f:M\longrightarrow N$ is a (simple)
b-fibration and $Y$ is an interior p-submanifold of a boundary
hypersurface of $N$ then $f^{-1}(Y)$ is the union of a p-clean collection,
$X_*,$ in $M$ and, for any order of blow-up, $f$ lifts to a b-fibration 
\begin{equation}
\tilde f:[M;X_*]\longrightarrow [N,Y].
\label{GP.332}\end{equation}
\end{proposition}

\begin{proof} All the blow-ups are near the boundary hypersurface
  $H\in\cM_1(N)$ containing $Y$ or its preimage in $M.$ We may therefore
  localize, replacing $N$ by 
\begin{equation}
H\times[0,\epsilon )_x
\label{GP.333}\end{equation}
where $x$ is a boundary defining function for $H.$ This gives a local
extension of $Y$ to 
\begin{equation}
\tY=Y\times[0,\epsilon ),\ \epsilon >0,
\label{GP.334}\end{equation}
which is an interior p-submanifold with $Y=H\cap\tY$ one of its boundary
hypersurfaces. The preimage of $Y$ is therefore the part of the preimage
$f^{-1}(\tY)$ which meets one of the boundary hypersurfaces
$H_i'\in\cM_1(M)$ mapping to $H.$ 

Thus the preimage of $Y$ under $f$ is a collection of the boundary
hypersurfaces of $f^{-1}(\tY).$ It follows from Lemma~\ref{GP.330} that the
blow-up of the $X_*$ is permissible in any order and once an order is
chosen, $[M,X_*]$ is well-defined.

To show that $f$ lifts to b-fibration as in \eqref{GP.332} we first show
that it lifts to a smooth map. This is a local statement near the preimage
in $[M,X_*]$ of an arbitrary point $m'\in f^{-1}(Y)$ with image $m\in Y.$
Now, near $m$ we may choose adapted local coordinates $y_j,$ $j=1,\dots,n,$
$x'_l,$ $y'_r$ in $H$ such that $Y$ is defined by $y=0,$ and the $x'_l$
define the boundaries hypersurfaces of $H.$ Since $f$ is a b-fibration it
decomposes as a product near $m'$ in terms of local coordinates $\xi_i,$
$y,$ $y'$ and boundary defining functions $\xi_p'$ so locally
\begin{equation}
f(\xi,y,\xi',y',z)=(\xi_1\dots\xi_k,y,g(\xi',y'))
\label{GP.335}\end{equation}
where $g$ is a b-fibration in the remaining variables (which can also be
brought to normal form) and the $z$ are interior fibre variables. All the
blow-ups are in the variables $\xi_i$ and $y,$ defining the various $X_i$
locally, so we can ignore the variables $\xi',$ $y'$ and $z$ and suppose
that $Y$ is a point in $H$ which is simply a ball in $\bbB^n_y$ and we are
considering the map 
\begin{equation}%
\begin{gathered}
f:\{(\xi,y)\in[0,1)^k\times\bbB^n;\xi_1\dots\xi_k<1\}\\
f(\xi,y)=(x,y),\ x=\xi_1\dots\xi_k.
\end{gathered}
\label{GP.336}\end{equation}

Consider the lift of a quadratic defining function for $Y,$ 
\begin{equation}
f^*(x^2+|y|^2)=\xi_1^2\dots\xi_k^2+|y|^2.
\label{GP.337}\end{equation}
The blow-up of $X_1=\{\xi_1=0,y=0\}$ introduces the smooth functions 
\begin{equation*}
\eta_1=(\xi_1^2+|y|^2)^{\ha},\
\omega _i=\frac{y_i}{\eta_1},\ \Xi_1=\frac{\xi_1}{\eta_1}
\label{GP.338}\end{equation*}
which generate the \ci\ structure of the blow-up. The lift of $\tY$ is now
$\omega =0$ and 
\begin{equation*}
\xi_1=\eta_1\Xi_1,\ f^*(x^2+|y|^2)=\eta_1^2(\Xi_1^2\xi_2\dots\xi_k^2+|\omega|^2).
\label{GP.339}\end{equation*}
Here $\eta_1$ is a defining function for the front face and $\Xi_1$ is a
defining function for the lift of the boundary hypersurface $\xi_1=0$
containing $X_1.$ It follows that, away from the lift of $\tY,$  $f$ lifts
to a smooth map to the blow-up of $Y$ since $|\omega |^2>0$ and hence  
\begin{equation*}
f^*(r)=\eta_1(\Xi_1^2\xi_2\dots\xi_k^2+|\omega|^2)^\ha,\ f^*(y/r)=\frac{\eta_1}r
\omega.
\label{GP.340}\end{equation*}

Thus we can proceed iteratively blowing up the lift of the next $X_j$
\begin{equation*}
\{\xi_j=0,\ \omega ^{(j-1)}=0\},\ \Mwhere
f^*(r^2)=\eta_1^2\dots\eta_{j-1}^2(\Xi_1^2\dots\Xi_{j-1}^2\xi_j^2\dots\xi_k^2+
|\omega^{(j-1)}|^2)
\label{GP.341}\end{equation*}
with essentially the same argument justifying the smoothness away from the
remaining $X_*.$ After the final step 
\begin{equation*}
f^*(r^2)=\eta_1^2\dots\eta_{j-1}^2(\Xi_1^2\dots\Xi_{k}^2+|\omega^{(k)}|^2).
\label{GP.344}\end{equation*}
Here the final lift of $\tY$ is $\{\omega ^k=0\}$ but this p-submanifold is
disjoint from all the lifted hypersurfaces $\{\Xi_i=0\}$ so the last factor
is strictly positive.

By induction it follows that $f$ lifts to a smooth map \eqref{GP.332} near
the arbitrary point $m'$ and
\begin{equation}
\begin{gathered}
\tilde f^*(r)=a\eta_1\dots\eta_k,\ a>0\text{ and smooth,}\\
\tilde f^*(x/r)=a'\Xi_1\dots\Xi_k,\ a'>0\text{ and smooth.}
\end{gathered}
\label{GP.343}\end{equation}
Thus $\tilde f$ is a b-normal b-map since these are the defining functions
for the boundary faces of $[N;Y].$

To see that $\tilde f$ is a b-fibration it suffices to show that it is a
b-submersion, i.e.\ any element of $\bV([N;Y])$ has a lift to an element of
$\bV([M;X_*]).$ These b-vector fields on $[N;Y]$ are generated by the lift
from $N$ of  
\begin{equation}
\cW=\{W\in\bV(N);W\text{ is tangent to }\tY\}+x\bV(N).
\label{GP.346}\end{equation}
This can be seen, after the simplifications discussed above, from the fact
that these reduce in local coordinates to  
\begin{equation*}
x\pa_x,\ y_j\pa_{y_l},\ x\pa_{y_j}\Mand x^2\pa_x
\label{GP.347}\end{equation*}
for all $j$ and $l.$ Under the iterated blow-up of the $X_i$ these lift to
be smooth, so it follows that $\tilde f$ is a b-fibration.
\end{proof}

\section{Resolution in size order}\label{Res-size}

In the sequel we encounter iterated blow-ups. As noted in the Introduction
several of the results of thus section are extensions of unpublished
work with Michael Singer.

\begin{definition}\label{GP.360} A closed subset, $Q\subset M,$ is
  \emph{liftable} under the blow-up of a p-submanifold, $C,$ if it is
  either contained in $C$ or else $Q\setminus C$ is dense in $Q.$ The lift
  (or proper transform) of $Q,$ $l(Q)\subset [M;C],$ is defined to be $\beta
  ^{-1}(C)$ in the first case and the closure of $\beta ^{-1}(Q\setminus
  C)$ in the second.
\end{definition}
\noindent So in case $Q\subset C,$ $l(Q)\subset\ff_C$ and in the second it
is the closure of $\beta ^{-1}(Q)\setminus\ff_C.$ In either case $\beta (l(Q))=Q.$

Under the blow-up of one p-submanifold of a p-clean collection any other
element is liftable and the lift is a p-submanifold of the blown up manifold.

\begin{lemma}\label{GP.356} If $C$ and $D$ are p-submanifolds of $M$
  intersecting p-cleanly and $C$ is not contained in $D$ then the lift of
  $C$ under the blow-up of $D$ is naturally diffeomorphic to $[C;C\cap D].$
\end{lemma}

To proceed with an iterated blow up it needs to be checked that
successive centres are p-submanifolds when the preceding ones have been
blown up. If $\FC$ is a p-clean collection of submanifolds of $M$ with a
given order we define successive blow-ups inductively. 

\begin{definition}\label{GP.183} An order of a p-clean collection of
  submanifolds is \emph{permissible for resolution} if, at succesive steps
  in the blow up, the next element is liftable through each preceding
  blow-up and the final lift is a p-submanifold. We call the manifold
  resulting from the permissible blow up of a p-clean collection a
  \emph{resolution} (of the family or manifold). Two permissible orders are
  equivalent if the identity on the complement of the union of the centres
  extends to a diffeomorphism of the resolutions.
\end{definition}

\begin{lemma}\label{GP.330} If $Y\subset M$ is an interior p-submanifold of
  a manifold with corners then any order on the boundary hypersurfaces of
  $Y$ is permissible for resolution and at each stage the remaining boundary
  hypersurfaces lift to be a p-clean family.
\end{lemma}

\begin{proof} The boundary hypersurfaces of an interior p-submanifold
  certainly form a p-clean collection since locally, near any point, $Y$ is
  defined by the vanishing of some tangential coordinates, $y_i,$
  $i=1,\dots,q.$ Under the blow up of one boundary hypersurface, $Y$ lifts
  to again be an interior p-submanifold and the remaining boundary
  hypersurfaces lift to be boundary hypersurfaces of the lift. Thus the
  result follows by induction.
\end{proof}

In particular for collections of boundary faces any order of blow-up is
permissible. However the final result may depend on the order.

In general the lifts of the remaining elements in a p-clean collection
under the blow-up of one element, although well-defined, in the steps of a
permissible resolution do not form a p-clean collection. This can be seen
from the (relevant) example of three lines all lying in a plane, and their
intersection, in $\bbR^3:$
\begin{multline}
F_1=\{x_1=0=x_3\},\ F_2=\{(0,0,0)\},\\ F_3=\{x_2=0=x_3\},\ F_4=\{x_1=x_2,
x_3=0\}. 
\label{GP.324}\end{multline}
After the blow-up of $F_1$ there is no simultaneous product
decomposition. This is restored by the blow up of $F_2.$ This is an example
of an intersection order as discussed below.

\begin{definition}\label{GP.319} 
A \emph{size order} for a p-clean collection is an order in which the
dimension is (weakly) increasing.
A p-clean collection, $\FC,$ of submanifolds of $M$ is intersection-closed
(or a \emph{pic}) if 
\begin{equation}
C_1,\ C_2\in\FC\Longrightarrow C_1\cap C_2=\emptyset\Mor C_1\cap C_2\in\FC.
\label{GP.358}\end{equation}
\end{definition}
\noindent Clearly a size order always exists.

Note that there is an issue to do with connectedness. Generally
submanifolds are taken to be connected but then the intersection of two
submanifolds need not be connected. To avoid pedantry we shall henceforth
allow manifolds to have a finite number of connected components all of the
same dimension but think of then as such, i.e. finite unions of connected
manifolds. So the condition in \eqref{GP.358} can be written $C_1\cap
C_2\subset\FC,$ meaning each component is an element. In this sense we
think of submanifolds with several components as `collective submanifolds'.

With this proviso on connectedness, pics have natural functorial
properties. If $B$ is a boundary face of $M$ or $C\in\FC$ then the
collections 
\begin{equation}
\{B\cap C',\ C'\in\FC\}\Mand \{C\cap C',\ C';\in\FC\}
\label{GP.359}\end{equation}
are pics, written $B\cap\FC$ and $C\cap\FC$ (in $C)$ respectively.

\begin{proposition}\label{GP.82} Under the blow-up of an element of minimal
  dimension of a pic the lifts of the remaining elements form a pic.
\end{proposition}

\begin{proof} Since no other element can be contained in an element of
  minimal dimension, $F,$ the lifts of the other elements are the closures
  of the inverse images of the complements of $F$ under the blow-down
  map. The condition that these form a p-clean collection is everywhere
  local and certainly satisfied away from the front face of the blow-up
  where they are unchanged. Consider a point $p$ in the front face and
  coordinates $x'_i,$ $x''_j,$ $y'_l,$ $y''_p,$ based at the image $\beta
  (p)$ under the blow-down map where $F$ is locally defined by $x'=0,$
  $y'=0.$ If an element of the p-clean collection contains $\beta(p)$ then
  it must contain $F$ locally since otherwise the intersection with $F$
  would be smaller. Thus we can suppose that the other elements through
  $\beta (p)$ are each defined by some subset of the $x'_i$ and linear
  functions of the $y'_l.$ One of the $x'_i$ or $y'_l,$ denoted $r,$
  dominates the others near $p$ and it is then a defining function for the
  front face and the remaining $x'_i/r$ and $y'_l/r$ are coordinates on the
  front face near $p.$ Consider the elements which lift to contain $p.$
  They are transversal to the front face so $dr\not=0$ on them and they are
  each defined locally by some of the $x'_i/r$ and some linear functions of
  the $y'_l/r.$ By subtracting the values at $p$ we may renormalize the
  (non-trivial) $y'_l/r$ to give $Y'_l$ which vanish at $p.$ Then the
  lifted manifolds are defined by the appropriate $x'_r/r$ and affine
  functions of the $Y'_l.$ To vanish at $p$ the affine functions must be
  linear. This shows that the lifted collection is p-clean.

To see that the lifted collection is closed under intersection it suffices
to note that the lift of the intersection of any two elements is the
intersection of the lifts except in case two manifolds intersect in $F;$
then their lifts are disjoint.
\end{proof}

\begin{corollary}\label{GP.361} If $\FC$ is a pic in a compact manifold
  with corners, $M,$ then any size order is permissible and the resulting
  resolution, denoted $[M;\FC]_{\le},$ or simply $[M;\FC],$ is independent
  of choice.
\end{corollary}

\begin{proof} Certainly the elements are partially ordered by dimension, so a
  size order is simply a total order compatible with this partial
  order. The elements of minimal dimension are necessarily disjoint and
  hence can be blown up in any order. 

Thus the iterative blow-up is well-defined and the only choice at each
level is the irrelevant one of the order of blow-up of the (at the stage of
blow up) disjoint elements of minimal dimension.
\end{proof}

It is useful to have a description of the resolved manifold.

\begin{proposition}\label{GP.357} The (size-ordered) resolution $\beta
:[M;\FC]\longrightarrow M$ of a pic in
a compact manifold has (collective) boundary hypersurfaces labelled by
$\cM_1(M)\cup\FC$ and denoted respectively $\beta ^{\#}(H)$ and $\ff_C,$ for
$H\in\cM_1(M)$ and $C\in\FC,$ with $\beta ^{\#}(H)=[H;\FC\cap H]$ and more
generally $\beta :\ff_C\longrightarrow C$ being the composite $\beta _C\phi
_C$ of a fibration $\phi _C:\ff_C\longrightarrow [C;\FC\cap C],$ with fibre
a resolution of a fractional sphere, and the blow-down map for $[C;\FC\cap C].$
\end{proposition}

\begin{proof} Applying Lemma~\ref{GP.356} repeatedly gives the first part
  of the result. The description of $\ff_C,$ for $C\in\FC,$ follows
  similarly. The blow-up of all the elements of $\FC$ which are smaller
  than $C$ replaces it by $[C;\FC\cap C]$ and giving the blow-down map
  $\beta_C.$ Blowing up the lift of $C$ up gives an initial
  fibration $\phi'_C.$ Subsequent (larger) elements of $\FC$ either lift to
  be disjoint from the lift of $C$ or originally contained it. The latter
  form a pic consisting of sections of the fibration $\phi'_C$ so these
  blow-ups replace $\phi'_C$ by a final fibration
  $\phi_C:\ff_C\longrightarrow [C;\FC\cap C]$ with resolved fibres.
\end{proof}

\section{Resolution in intersection order}\label{Res-int}

\begin{definition}\label{GP.320} An \emph{intersection order} on a pic is
an order such that for any two elements $F<G$ with non-empty intersection
$F\cap G\le G.$
\end{definition}

Clearly a size order is an intersection order. An intersection order is a
size order precisely when no element is preceded by a larger one. Thus the
integer
\begin{equation}
\mu(o)=\sum\limits_{F\le_o G}(\dim F-\dim G)_+
\label{GP.366}\end{equation}
vanishes only for size orders. This leads to a simple `path' from an
intersection order to a size order with each step an exchange of
neighbours.

Namely for an intersection order on a give pic, $\FC,$ consider the first
element $F_i$ such that $\dim F_{i+1}<\dim F_i.$ Then the order $o'$
  obtained from $o$ by exchange of $F_i$ and $F_{i+1},$ so 
\begin{equation}
F_{i+1}<_{o'}F_i,\ F<_oG\Longrightarrow F<_{o'}G\text{ otherwise}
\label{GP.367}\end{equation}
is a new intersection order with $\mu(o')<\mu(o).$ After a finite
number of steps this leads to a size order. Note that for the pair involved
in passing from $o$ to $o'$ 
\begin{equation}
F_i\cap F_{i+1}=\emptyset\Mor F_i\cap F_{i+1}<F_i\Mor F_{i+1}\subset F_i.
\label{GP.368}\end{equation}

\begin{lemma}\label{GP.362} Suppose $C_i\subset M,$ $i=1,2$ are
  p-submanifolds which either intersect transversally or else $C_1\subset
  C_2$ then if $Q\subset M$ is a closed subset liftable under blow up of
  $C_1$ with lift, $l_1(Q),$ liftable under blow up of $l_1(C_2)$ then $Q$ is
  liftable under blow up of $C_2$ with lift $l_2(Q)$ liftable under blow up
  of $l_2(C_1)$ and the two double lifts are the same.
\end{lemma}

\begin{proof} Transversal intersection includes the case that $C_1$ and
  $C_2$ are disjoint, in which case the result is obvious.

If the intersection is non-empty and transversal it suffice to consider a
small neighbourhood of point of $Q.$ If this is not in $C_1\cap C_2$ then
again this commutativity result holds trivially. Take a neighbourhood of a
point in the intersection which is the product of three balls $B_1\times
B_2\times B$ where there are points $p_i\in B_i$ such that locally
\begin{equation}
C_1=\{p_1\}\times B_2\times B,\ C_2=B_1\times \{p_2\}\times B.
\label{GP.363}\end{equation}
Then the blow-up of $C_1$ is $[B_1,\{p_1\}]\times B_2\times B$ and
$l_1(C_2)=[B_1,\{p_1\}]\times\{p_2\}\times B.$

In case $Q\subset C_1$ it is of the form $\{p_1\}\times F$ with lift
$l_1(Q)=\ff([B_1,\{p_1\}])\times F.$ This is liftable for the blow-up of
$l_1(C_2)$ if $l_1(Q)\subset l_1(C_2),$ which is equivalent to
$F=\{p_2\}\times F'$ so $Q\subset C_1\cap C_2$ and 
\begin{equation*}
l_2(l_1(Q))=\ff([B_1,\{p_1\}])\times\ff([B_2,\{p_2\}])\times F',
\label{GP.364}\end{equation*}
which is symmetric so the result follows. The other possibility, given that
$Q\subset C_1,$ is that $(l_1(Q)\setminus\ff([B_1,\{p_1\}])\times
\{p_2\}\times B$ is dense in $l_1(Q)$ which is equivalent to the density of
$F'=F\setminus(\{p_2\}\times B)$ in $F$ and hence of $Q\setminus C_2$ in $Q.$
Thus $Q$ is liftable for the blow-up of $C_2$ with lift $\{p_1\}\times
\tilde F$ where $\tilde F$ is the closure of $F'$ in $[B_2\times
  \{p_2\}]\times B.$ This is contained in $l_2(C_1)$ and the double lifts
are the same.

Next consider the case, as in \eqref{GP.363}, that $Q\setminus C_1$ is dense
in $Q.$ Then $l_1(Q)\subset l_1(C_2)$ if and only if $Q\subset C_2$ and the
result follows readily. In the remaining case $Q\setminus(C_1\cup C_2)$ is
dense in $Q$ and again the result follows. This completes the case that
$C_1$ and $C_2$ meet transversally.

Now we may suppose $C_1\subset C_2.$ If $Q\setminus C_1$ is dense in $Q$
then $l_1(Q)\subset l_1(C_2)$ if and only if $Q\subset C_2.$ Thus $l_1(Q)$
is liftable under blow up of $C_2$ with lift the union of the fibres over
$Q.$ The $l_2(Q)\setminus l_2(C_1)$ is the union of the fibres over
$Q\setminus C_1$ so is dense in $l_2(Q)$ and the result follows. If
$l_1(Q)\setminus l_1(C_2)$ is dense in $l_1(Q)$ then $Q\setminus C_2$ is
dense in $Q$ and both double lifts are defined and equal to the closure of
the double lift of $Q\setminus C_2.$ In case $Q\subset C_1\subset C_2$ all
lifts are preimages and commutativity follows.
\end{proof}

\begin{lemma}\label{GP.321} For any intersection order on a pic, $\FC,$ and
  any $G\in\FC$
\begin{equation}
\FC_{G}=\{F\in\FC;F\le G\}
\label{GP.322}\end{equation}
is closed under intersection.
\end{lemma}

\begin{proof} Clearly this is true when $G$ is the first element so we can
  proceed by induction over the order, assuming that $\FC_K$ is closed
  under intersection for all $K<G.$ Then it suffices to show that $F\cap
  G\in\FC_G$ if $F<G$ which is Definition~\ref{GP.320}.
\end{proof}

\begin{theorem}\label{GP.323} Any intersection order on a pic, $\FC,$ in
  $M$ is permissible and the resulting resolution is canonically
  diffeomorphic to $[M;\FC]$ from a size order, i.e.\ all intersection
  orders are equivalent for resolution.
\end{theorem}

\begin{proof} We proceed by induction over the number of elements of the
  collection. To prove permissibility of blow-up we may assume the result
  for the collection consisting of the elements preceding the last using
  Lemma~\ref{GP.321}. We start by considering the special intersection
  given by a size order on all the elements preceing the last element, $G.$
  Thus $G$ has some position in a size order on the whole collection. The
  elements of strictly smaller dimension than $G$ may then be blown up in
  this size order and the lift of $G$ is one of the minimal remaining
  elements. The initial intersection of each of the remaining elements is
  no larger than $G.$ If this intersection is strictly smaller than $G$
  then it has already been blown up and the lift of the element is disjoint
  from $G.$ Thus the remaining lifted elements which meet $G$ are those
  which initially contained it. In fact they, without $G,$ form a pic since
  the intersection of any two must, in the initial intersection order,
  precede one of them and hence strictly precede $G.$ Thus at this point in
  the resolution the remaining elements which intersect $G$ form a pic with
  a unique minimal element $H.$ Proceeding from the complete blow up in
  size order, $G$ and $H$ may be interchanged, by commutativity. When $H$
  is blown up all the other (relevant) lifted centres meet the front face
  it produces in sections of the fibration over $H$ and the lift of $G$ is
  its inverse image under this blow-up. It follows that the lift of $G$ and
  each of the remaining elements is transversal. Thus $G$ can be commuted
  back to the last position showing that this special intersection order is
  permissible for resolution and equivalent to a size order.

To proceed further we alter the size order on the elements preceding $G$
step by step, using \eqref{GP.367} (in reverse) to successively exchange
neighbours. Thus we can again assume inductively that up to some point
these resolutions are all permissible and the elements preceeding the pair
exchanged are in size order. Then again it follows that the lifts through
the blow-up of these earlier elements leave the lifts of $F_i$ and
$F_{i+1}$ either transversal (including disjoint) or with the larger
included in the smaller. Then Lemma~\ref{GP.362}, applied with $Q$ the lift
of $G,$ proves the inductive step showing that the new order is also
permissible.
\end{proof}

\begin{corollary*}\label{GP.182} Suppose $\FC$ is a p-clean collection
  which can be written 
\begin{equation}
\FC=\FG\sqcup \FR
\label{GP.254}\end{equation}
where $\FG$ and $\FR$ are separately closed under intersection and the
intersection of an element of $\FG$ with an element of $\cR$ lies in $\cR,$
then an order in which the $\FG$ appear first and then the $\FR,$ both in
a size order, is permissible and results in the same space 
\begin{equation}
[M;\FC]=[[M;\FG];\FR].
\label{GP.255}\end{equation}
\end{corollary*}

\begin{proof} This is an intersection order.
\end{proof}

\section{Diagonals and multidiagonals}\label{MD}

Consider a compact generalized product, where the manifolds $M[k]$ are all
supposed connected. The (multi)diagonals in $M[k]$ are labelled by
partitions, $\mathfrak{P},$ of $J(k)$ corresponding to disjoint subsets
$P_l,$ $l=1,\dots,L(\mathfrak{P}),$ each with at least two
elements, the remainder being singletons. The corresponding diagonal
\begin{equation}
D_{\mathfrak{P}}=D_{\mathfrak{P}}^{(k)} \subset M[k]
\label{GP.206}\end{equation}
is the image of a p-embedding. In general there are several choices of
surjective map $I:J(k)\longrightarrow J(k_{\mathfrak{P}}),$ where
$k_{\mathfrak{P}}$ is $L(\mathfrak{P})$ plus the number of singletons such
that $D_{\mathfrak{P}}$ is the image of
$S_{I}:M[k_{\mathfrak{P}}]\longrightarrow M[k].$ It is convenient to make a
specific choice. Namely we take
\begin{equation}
I_{\mathfrak{P}}:J(k)\longrightarrow J(k_{\mathfrak{P}})
\label{GP.207}\end{equation}
to be the map which sends each of the sets $\cP_{j},$
determining $\mathfrak{P},$ to $j$ and is strictly increasing on the
singletons. The associated p-embedding will also be denoted $D_{\mathfrak{P}}.$

The functorial properties of the generalized product show that
the image $D_{\mathfrak{P}}$ under any of the structural diffeomorphisms of
$M[k],$ forming a representation of the permutation group, is another
$D_{\mathfrak{P}'}.$

Under the partial order of partitions, where $\mathfrak{Q}<\mathfrak{P}$ if every
$Q_j$ is contained in some $P_l,$
\begin{equation}
\mathfrak{Q}<\mathfrak{P}\Longrightarrow D_{\mathfrak{P}}\subset D_{\mathfrak{Q}}
\label{GP.208}\end{equation}
since the surjection defining $D_{\mathfrak{P}}$ can be factored through
that defining $D_{\mathfrak{Q}}.$

Consider the diagonals, $D_{i,j},$ corresponding to a partition with one
non-trivial subset containing just two distinct element, $i$ and $j.$ For
any partition $\mathfrak{P}$ there is a structural diffeomorphism mapping
$D_{\mathfrak{P}}$ to $D_{\mathfrak{Q}}$ where the elements of each
$\mathfrak{Q}_j$ are successive integers in $J(k)$
\begin{equation}
Q_j=\{r_j,r_j+1,\dots,r_j+k_j\}.
\label{GP.209}\end{equation}
It follows by functoriality that
\begin{equation}
D_{\mathfrak{Q}}\subset\bigcap_{\{i;i,i+1\in Q_j\}} D_{i,i+1}
\label{GP.210}\end{equation}
where subsets of $J(k)$ with at least two elements can be identifed with
partitions. Note that the corresponding $D_{P}$ are the `true' diagonals and
we think of $D_{i,j}$ corresponding to $P=\{i,j\}$ as a `simple
diagonal'. Indeed, the map $J(k)\longrightarrow J(k')$ defining the
diagonal can be factored through the $J(k)\longrightarrow J(k-1)$ defining
any one of the simple diagonals on the right in \eqref{GP.210}.

We proceed to show that there is equality in \eqref{GP.210} (which of
course is trivial unless $k\ge3)$ but first consider a single simple
diagonal. Let $\Pi_{i,j}:M[k]\longrightarrow M[2],$ for $i<j\le k,$ be the
structural b-fibrations corresponding to the injective map
$J(2)\longrightarrow J(k)$ taking $1$ to $i$ and $2$ to $j.$

\begin{lemma}\label{GP.212} The diagonal $D_{i,j}$ in $M[k]$ for $i<j$ and
  $k\ge3$ satisfies
\begin{equation}
D_{i,j}=\Pi_{i,j}^{-1}(D).
\label{GP.202}\end{equation}
\end{lemma}

\begin{proof} The basic diagonal $D\subset M[2]$ is the image of the
  embedding corresponding to the unique map $J(2)\longrightarrow J(1).$ For
  any $i<j$ the functorial properties show that 
\begin{equation}
\Pi_{i,j}D_{i,j}=DS_I,\ I:J(1)\longrightarrow J(k-1),\ I(1)=i.
\label{GP.203}\end{equation}
Thus it follows that as submanifolds $\Pi_{i,j}(D_{i,j})\subset D.$ The
inverse image, $\Pi_{i,j}^{-1}(D),$ of an interior p-submanifold under a
b-fibration is an interior p-submanifold and it follows that the dimension
is $\mu +(k-2)\kappa.$ Since this is the same as that of $M[k-1]$ and hence of
the submanifold $D_{i,j},$ $D_{i,j}$ must be a component of
$\Pi_{i,j}^{-1}(D).$ Since we are requiring the connectivity of the fibres of
b-fibrations \eqref{GP.202} holds.
\end{proof}

\begin{lemma}\label{GP.201} The p-submanifolds $D_{1,2}$ and $D_{1,3}$ in
  $M[k]$ for $k\ge3$ intersect in the diagonal $D_{P}$ corresponding to the single
  subset $P=\{1,2,3\}$ in $J(k)$ and the intersection is transversal.
\end{lemma}

\noindent The conormal directions here can be thought of as in the dual to the
b-tangent space but since all the manifolds are interior p-submanifolds,
this is the same as the ordinary conormals.

\begin{proof} Consider the b-fibration $\Pi^{k}_{3}:M[k]\longrightarrow
  M[k-1]$ corresponding to the increasing injection $J(k-1)\longrightarrow
  J(k)$ with image omitting $3.$ By functoriality the image
  $\Pi^{k}_{3}(D_{1,2})$ is the corresponding simple diagonal
  $D_{1,2}\subset M[k-1].$ However $\Pi^{k}_{3}$ is a diffeomorphism of
  $D_{1,3}$ to $M[k-1].$ It follows that the intersection is transversal,
  since a fibre of $\Pi^{k}_{3}$ which is contained in $D_{1,2}$ is mapped
  into a point in $D_{1,2}$ so must be transversal to $D_{1,3}$ at the
  intersection. Moreover a point in $D_{1,3}$ is uniquely of the form
  $D_{1,3}(m)$ with $m\in M[k-1]$ and $\Pi^{k}_{3}(D_{1,3}(m))=m.$ Thus if
  $D_{1,3}(m)$ is in the intersection, i.e.\ $D_{1,3}(m)\in D_{1,2},$ it
  follows that $\Pi^{k}_{3}(m)\in D_{1,2}\subset M[k-1].$ The intersection
  is therefore the image of $D_{1,2}\subset M[k-1]$ under $D_{1,3}$ and this
  is functorially identified with the diagonal $D_{P}\subset M[k]$ for
  $P=\{1,2,3\}.$
\end{proof}

\begin{lemma}\label{GP.213} For any $3\le l\le k$ the intersection 
\begin{equation}
D_{1,2}\cap D_{2,3}\cap\dots\cap D_{l-1,l}\Min M[k]
\label{GP.214}\end{equation}
is transversal and equal to the diagonal corresponding to the single subset
$P=\{1,\dots,l\}\subset J(k).$ 
\end{lemma}

\begin{proof} Lemma~\ref{GP.201} is this statement for $l=3.$ We proceed
  iteratively in $l$ assuming the result for $l<L$ and all $k.$ Since the
  structural diffeomorphisms can be used to shift the indices we can assume
  that transversality holds for 
\begin{equation*}
D_{1,2}\cap\dots\cap D_{L-2,L-1}\Min M[k-1]
\label{GP.215}\end{equation*}
with the intersection identified with $D_{P},$ $P=\{1,\dots,L-1\}.$ The
b-fibration $\Pi_{1}^{k}:M[k]\longrightarrow M[k-1],$ corresponding to the
increasing injective map from $J(k-1)$ to $J(k)$ omitting $1,$ maps
$D_{1,2}$ onto $M[k-1]$ and each $D_{j,j+1}$ for $j>1,$ onto
$D_{j-1,j}\subset M[k-1].$ From this we deduce the inclusion of the left in
the right in
\begin{equation}
D_{2,3}\cap\dots\cap
D_{L-1,L}=\Pi^{k}_{L}\left (D_{1,2}\cap\dots\cap D_{L-2,L-1}\right)\Min M[k].
\label{GP.216}\end{equation}
Equality follows from the equality of dimensions \eqref{GP.202}. Since the
fibres of $\Pi^{k}_{L}$ at a point of the intersection \eqref{GP.214}
contained in $D_{1,2}$ cannot be tangent to $D_{2,3}\cap\dots\cap
D_{l-1,l}$ transversality follows, as does the identification of the intersection.
\end{proof}

\begin{theorem}\label{GP.217} The diagonals in a compact (connected)
  generalized product form a p-clean collection closed under
  intersection. The `true diagonals' $D_{P}$ for single subsets
  $P\subset J(k)$ form a p-clean collection closed under
  non-transversal intersection.
\end{theorem}

\begin{proof} Consider a diagonal $D_{P}$ in $M[k].$ Using the structural
  diffeomorphisms it suffices to assume that $P=\{1,2,\dots,l\},$ $2\le l\le
  k.$ Then $D_{P}\subset D_{i,j}$ if $\{i,j\}\subset P$ and conversely
  the diagonals containing $D_{P}$ are those corresponding to
  partitions contained in $P.$

We know that the b-fibration $\Pi_{1,j}:M[k]\longrightarrow M[2],$ $2\le
j\le l,$ restricts to a b-fibration of $D_{1,j}$ over $D$ and that 
\begin{equation}
D_{P}=D_{1,2}\cap D_{1,3}\dots\cap D_{1,l}\text{ is transversal.}
\label{GP.220}\end{equation}
Moreover the image of a point in $D_{P}$ under each of the $\Pi_{1,j},$
$2\le j\le l,$ is the same. Since, by functoriality,
$\Pi_{1,j}R_{1,j}=R\Pi_{1,j}$ and (hence) $D$ is a component of the
fixed point set of $R$ near a point in $D.$ So we may choose $\kappa$
independent functions, denoted collectively $z,$ each odd under $R,$ such that
$D=\{z=0\}$ locally. Then the functions
\begin{equation}
z_j=(\Pi_{1,j})^*z,\ j=2,\dots,l\Msatisfy
R_{1,j}^*z_j=-z_j,\ D_{1,j}=\{z_j=0\}\text{ locally.}
\label{GP.218}\end{equation}
By Lemma~\ref{GP.213} their differentials must (all) be independent. 

Now consider the other $D_{i,r}$ where $1, i<r\le l.$ The structural identity
\begin{equation}
\Pi _{1,i}R_{i,r}= \Pi _{1,r}\Longrightarrow (R_{ir})^*(z_i-z_r)=z_r-z_i.
\label{GP.219}\end{equation}
Since the components of $z_r-z_i$ are linearly independent it follows that
\begin{equation}
D_{i,r}=\{z_i-z_r=0\}\text{ locally.}
\label{GP.221}\end{equation}

Thus all the diagonals containing one $D_{P}$ for
  $P\subset J(k)$ form a p-clean collection. The general case
  follows from this since if $\mathfrak{P}$ consists of several disjoint sets
  $P_j$ then  
\begin{equation}
D_{\mathfrak{P}}=\bigcap_{i}D_{P_i}
\label{GP.222}\end{equation}
is transversal and the local linear defining functions \eqref{GP.221} can be
constructed independently.
\end{proof}

Let $\cD(k)$ be the collection of multidiagonals in $M[k]$ for a
generalized product. Thus $D_{\FP}\in\cD(k)$ corresponds to a partition
$\FP$ of $J(k).$ Then for $k>1$  
\begin{equation}
\cD(k)=\cF\sqcup\cG
\label{GP.349}\end{equation}
where $\cF$ corresponds to the partitions in which $k$ is a singleton. Thus
$D_{\FP}\in\cD(k)$ determines a unique $D_{\FP'}\in\cD(k-1)$ where $\FP'$
is the partition of $J(k-1)$ with $k$ dropped; conversely $D_{\FP'}$
determines $D_{\FP}.$ For $D_{\FQ}\in\cG$ there is a corresponding
$D_{\FQ'}$ obtained by dropping $k$ from the partition, but in general
$D_{\FQ'}$ does not determine $D_{\FQ}.$

\begin{lemma}\label{GP.350} For any generalized product, if
  $\Pi_k:M[k]\longrightarrow M[k-1]$ is the b-fibration corresponding to
  $I:J(k-1)\longrightarrow J(k),$ $I(i)=i,$ then in terms of \eqref{GP.222}  
\begin{equation}%
\begin{gathered}
\Pi_k^{-1}(D_{\FP'})=D_{\FP},\ D_{\FP}\in\cF,\\ \Pi_k:D_{\FP}\longrightarrow
D_{\FP'}\text{ is a b-fibration and
}\\
\Pi_k:D_{\FQ}\longrightarrow D_{\FQ'}\text{ is a diffeomorphism.}
\end{gathered}
\label{GP.351}\end{equation}
\end{lemma}

\begin{proof} This follows from the corresponding properties for diagonals
  discussed above.
\end{proof}

\section{The Lie algebroid}\label{LA}

As noted above a generalized product is associated to a Lie algebroid,
which it resolves. Since it is a b-fibration the null spaces of the
differentials $(\Pi_R)_*\subset\bT M[2]$ form a vector bundle.

\begin{definition}\label{GP.300} The vector bundle $E$ over $M[1]$ is
  defined as the (abstract) pull-back
\begin{equation}
E=D^*\nul((\Pi_R)_*).
\label{GP.302}\end{equation}
\end{definition}

\begin{theorem}\label{GP.304} The space $\cE$ of smooth sections of $E$
  over $M[1]$ is a Lie algebroid with anchor map
\begin{equation}
(\Pi_L)_*:\cE\longrightarrow \CI(M[1];\bT).
\label{GP.299}\end{equation}
\end{theorem}

\begin{proof} That \eqref{GP.299} is a smooth bundle map follows from the
  fact that $\Pi_L$ is a b-fibration.

The main part of the proof is to see that $\cE$ has a natural Lie algebra
structure. To do so we follow, more abstractly, the discussion in
Examples~\ref{GP.294} and \ref{GP.156} above. Take an element $V\in\cE,$
i.e.\ a smooth section of $E$ over $M[1]$ or equivalently a smooth section
of $\nul((\Pi_R)_*)$ over the diagonal $D\subset M[2].$

The projection $\Pi_S$ restricts to a b-fibration of $D_{1,2}$ onto $D.$
Moreover the identity  
\begin{equation}
\Pi_R\Pi_S=\Pi_L\Pi_F\Longrightarrow
(\Pi_S)_*:\nul((\Pi_F)_*)\longrightarrow \nul((\Pi_R)_*).
\label{GP.305}\end{equation}
At $D_{1,2}$ the null space of $(\Pi_S)_*$ is tangent to $D_{1,2}$ whereas
$\nul((\Pi_F)_*)$ is transversal to $D_{1,2}.$ Since the dimensions are the
same the push-forward map in \eqref{GP.305} is therefore an isomorphism at
each point. It follows that $V$ lifts to a unique smooth section 
\begin{equation}
W\in\CI(D_{1,2};\nul((\Pi_F)_*))\text{ satisfying }(\Pi_S)_*W(p)=V(\Pi_S(p))
\label{GP.306}\end{equation}
at each point. 

Now, the projection $\Pi_C$ restricts to a diffeomorphism of $D_{1,2}$ onto
$M[2].$ Moreover, the identity 
\begin{equation*}
\Pi_R\Pi_F=\Pi_R\Pi_C
\label{GP.316}\end{equation*}
shows that the differential of $\Pi_C$ maps $\nul((\Pi_F)_*)$ into
$\nul((\Pi_R)_*)$ at each point of $D_{1,2}.$ Thus $W$ defined by
\eqref{GP.306} pushes forward to a smooth global section
\begin{equation}
\tilde V=(\Pi_C)_*W\in\CI(M[2];\nul((\Pi_R)_*)).
\label{GP.307}\end{equation}

This global section restricts to $D$ to be the original section $V.$
Indeed, on the triple diagonal, $T,$ $\Pi_S=\Pi_F=\Pi_C$ are all the
identity in terms of the embeddings of $M[1]$ as $T$ and $D.$ Also
\emph{at} the triple diagonal, $\nul((\Pi_F)_*)\subset TD_{2,3}.$ Now
$\Pi_CR_{2,3}=\Pi_S$ and $R_{2,3}$ acts as the identity on $D_{2,3}$ so it
follows that $(R_{2,3})_*W=W$ and hence $(\Pi_C)_*W=(\Pi_S)_*W=V$ at $D.$

The extension of $V$ to $\tilde V$ also gives an extension of $W$ to a
section $\tilde W\in\CI(\Omega ;\nul((\Pi_F)_*)$ at least for a small open
neighbourhood of the triple diagonal, $T.$ Indeed, in such a neighourhood
$(\Pi_C)_*$ maps $\nul((\Pi)_F)_*)$ isomorphically onto $\nul((\Pi_R)_*)$ so
$\tilde W$ is well-defined pointwise by
\begin{equation}
\tilde V=(\Pi_C)_*\tilde W\text{ near }T.
\label{GP.308}\end{equation}
Then $\tilde V$ and $\tilde W$ are $\Pi_C$-related smooth vector fields. 
The extended vector fields $\tilde V$ are tangent to the fibres of $\Pi_R$
so given two vector fields $V_1,$ $V_2\in\cE$ the commutators of the
extensions $[\tilde V_1,\tilde V_2]$ and $[\tilde W_1,\tilde W_2]$ are
$\Pi_C$-related as well. From this it follows that the lift of the
commutator defined on $\cE$ by restriction is the commutator of the
lifts. In particular, the Jacobi identity holds.

It also follows directly that the anchor map \eqref{GP.299} maps
commutators to commutators so $\cE$ does indeed have the structure of a Lie
algebroid.
\end{proof}

It is straightforward to check that the Lie algebroid structure on $\cE$ is
the same as that induced from the associative algebra of differential operators
formed by the pseudodifferential operators, discussed in the next section,
with kernels supported on the diagonal.

\begin{remark}\label{GP.309} The interior of $M[2]$ is a Lie groupoid
  quantizing the Lie algebroid $\cE.$ However in most interesting cases
  $M[2]$ itself is not a Lie groupoid. Rather, the intention is the notion
  of generalized product should correspond to `good' compactifications of
  certain Lie groupoids. Note that, as for a Lie group, it is totally
  unreasonable to expect that the compactification of a Lie groupoid should
  be a Lie groupoid.
\end{remark}

\section{Smoothing and pseudodifferential algebras}\label{Ops}

Each generalized product gives rise to `small' algebra of
pseudodifferential `operators' determined by kernels on $M[2]$ conormal
with respect to the diagonal. There are representations of these algebras,
acting on smooth functions on (or more generally corresponding algebras
acting on sections of a vector bundle over) $M[1].$ In the case of a
stretched product the kernels are determined by these actions as
operators. However, as is clear from Example~\ref{GP.156}, where the action
is trivial, this is not in general the case and there are `hidden
variables' corresponding to the fact that these pseudodifferential
operators always give a `microlocalization' of the Lie algebroid discussed
in \S\ref{LA}.

The algebra of fibrewise smoothing operators for a fibration
\eqref{GP.120}, with compact boundaryless fibres, arises directly from
\eqref{2.7.2023.1} and the basic fact that for any fibration of compact
manifolds, and any smooth complex vector bundle $W\longrightarrow Y$ over
the base the push-forward map, or fibre integration, gives a continuous
(and surjective) linear map
\begin{equation}
\phi_*:\CI(M;\Omega \otimes\phi^*W)\longrightarrow \CI(Y;\Omega\otimes W )
\label{2.7.2023.3}\end{equation}
where $\Omega$ is the density bundle on each respective space.

For any fibration there is an induced short exact sequence for the
differential 
\begin{equation}
T_{\phi}M\longrightarrow TM\overset{\phi_*}\longrightarrow TY,
\label{GP.121}\end{equation}
where $T_{\phi}M=\nul(\phi_*)$ is the fibre tangent bundle. This induces a
natural isomorphism for the associated density bundles 
\begin{equation}
\Omega (M)=\phi^*\Omega (Y)\otimes \Omega _{\phi}
\label{GP.122}\end{equation}
where $\Omega _{\phi}$ is the bundle of densities on the fibres of $\phi.$

The algebra of smoothing operators, on sections of a complex vector bundle
$V\longrightarrow M,$ is then identified as a space with
\begin{equation}
\Psi^{-\infty}_{\phi}(M;V)
=\CI(M^{[2]}_\phi;\Pi_L^*V\otimes\Pi_R^*V'\otimes\Pi_R^*\Omega_\phi).
\label{2.7.2023.4}\end{equation}
The action on $\CI(M;V)$ 
\begin{equation}
\Psi^{-\infty}_{\phi}(M;V)\times\CI(M;V)\longrightarrow \CI(M;V)
\label{2.7.2023.5}\end{equation}
comes from pushforward as in \eqref{2.7.2023.3} and \eqref{GP.122}. Namely for
$A\in\Psi^{-\infty}_{\phi}(M;V)$ and $u\in\CI(M;V)$ to define $Au$ we may
choose a positive global section $\nu\in\CI(M;\Omega)$ and set
\begin{equation}
Au=(\Pi_L)_*(\Pi_L^*\nu\cdot A\cdot\Pi_R^*u)/\nu.
\label{2.7.2023.6}\end{equation}
The product in \eqref{2.7.2023.6} includes the pairing of $\Pi_R^*V'$ and
$\Pi_R^*V$ and the tensor product identification of densities gives
$\Pi_R^*\nu\cdot A
\cdot\Pi_R^*u\in\CI(M^{[2]}_{\phi};\Pi_L^*V\otimes\Omega )$
  which pushes forward to define $(Au)\nu.$ That this is independent of the
  choice of $\nu$ follows from the fact that for \eqref{2.7.2023.3} in
  general $\phi_*(v\phi^*f)=f\phi_*(v)$ for all $f\in\CI(Y)$ and
  $v\in\CI(M;\Omega \otimes\phi^*W).$ 

The triple product corresponds to the composition of operators
\begin{equation*}
\xymatrix{
&M^2_{\phi}\\
&M^3_{\phi}\ar[dl]_{\Pi_S}\ar[dr]^{\Pi_F}\ar[u]^{\Pi_C}\\
M^2_{\phi}&&M^2_{\phi}.
}
\label{Newton2015.3}\end{equation*}
Then the kernel of the product is given by push-forward of densities
\begin{equation}
K_{AB}=(\Pi_C)_*\left(\Pi_S^*K_A\Pi_{1}^*\nu\times\Pi_F^*K_B\right)/
\Pi_{1}^*\nu
\label{GP.157}\end{equation}
where $\nu$ is again a positive smooth density on $M$ and the tensor
product pushed forward is naturally a density on $M^{[3]}_{\phi}.$ 

We proceed to generalize these constructions, especially the product, to a
compact generalized product $M[k]$ in place of the fibre products
$M^{[k]}_{\phi}.$

For a b-fibration of compact manifolds with corners, still denoted as in
\eqref{GP.120}, there is a short exact sequence generalizing \eqref{GP.121}
for the b-differential
\begin{equation}
{}^{\bo }T_{\phi}M\longrightarrow {}^{\bo}TM\overset{\phi_*}\longrightarrow
{}^{\bo} TY,
\label{GP.123}\end{equation}
so \eqref{GP.122} is replaced by  
\begin{equation}
\Omega_{\bo}(M)=\phi^*\Omega_{\bo} (Y)\otimes \Omega _{\bo,\phi}.
\label{GP.124}\end{equation}

The fibrations of a generalized product satisfy
\begin{equation*}
\Pi _R\Pi _S=\Pi _L\Pi _F
\label{GP.153}\end{equation*}
so
\begin{equation}
(\Pi _F)_*:\nul((\Pi _S))_*)\longrightarrow \nul ((\Pi _L))_*).
\label{GP.154}\end{equation}
This is an isomorphism near the triple diagonal.

\begin{definition}\label{GP.279} 
A generalized product is said to be \emph{regular} if the determinant of
\eqref{GP.154} is non-vanishing as a section 
\begin{equation}
\delta \in \rho ^\zeta \CI(M[3];\lambda ^d\nul ((\Pi _F)_*)\otimes\lambda ^d\nul((\Pi _L))_*))
\label{GP.155}\end{equation}
for some integer multiindex $\zeta.$ 
\end{definition}

In general the push-forward of a smooth b-density under a b-fibration is
\emph{not} smooth, in fact it may not exist with the obstruction to
existence being the possibility of `fixed hypersurfaces' of $M ,$
$H\in\cM_1(M)$ such that $\phi(H)=M.$ Apart from this, logarithmic terms in
general arise when two (intersecting) boundary hypersurfaces of $M$ are
mapped to one boundary hypersurface of $Y.$

\begin{definition}\label{GP.125} A smooth function (or section of a vector
  bundle) on a compact manifold with corners $Z_1$ is said to have
  \emph{simple support} with respect to a b-fibration $F:Z_1\longrightarrow
  Z_2$ if it vanishes to infinite order at the intersection of any two
  boundary hypersurfaces mapped to the same boundary hypersurface of $Z_2$
  by $F$ and also to infinite order at fixed hypersurfaces.
\end{definition}

\noindent Really only simple vanishing at the fixed hypersurfaces is
relevant for existence and smoothness of push-forwards.

\begin{proposition}\label{GP.126} If $F:Z_1\longrightarrow Z_2$ is a
simple b-fibration then fibre integration defines
\begin{equation}
F_*:\{u\in\CI(Z_1;F^*W\otimes\Omega );u\text{ has simple support}\}
\longrightarrow \CI(Z_2;W\otimes \Omega ).
\label{3.7.2023.1}\end{equation}
\end{proposition}

\begin{proof} This is a special case of the push-forward theorem from
  \cite{MR93i:58148}. 
\end{proof}

The usual way in which this support property is seen to be satisfied is for
$u$ to vanish to infinite order at all but one of the boundary
hypersurfaces of $Z_1$ in the preimage of each boundary hypersurface of
$Z_2.$ If $u$ vanishes to infinite order on all the hypersurfaces in the
preimage of a hypersurface of $Z_2$ then the push-forward also vanishes to
infinite order at that hypersurface.

The kernels of the smoothing `operators' associated to a compact
generalized product $M[k]$ are smooth sections of a bundle over $M[2].$ As
noted in Proposition~\ref{GP.82} each boundary hypersurface of $M[1]$
determines a diagonal boundary hypersurface of $M[2]$ which maps into it
under either of the projections $\Pi_L,$ $\Pi_R:M[2]\longrightarrow M[1].$

\begin{definition}\label{GP.130} The space of (small) smoothing kernels
for a vector bundle $V$ over $M[1]$ for a compact generalized product $M[*]$ is 
\begin{multline}
\Psi^{-\infty}(M[*];V)=\{A\in\CI(M[2];\Pi_L^*V\otimes\Pi_R^*V'\otimes
\Omega _{L});\\
A\equiv0\Mat H\in\cM_1(M[2])\Mif H\cap D=\emptyset\}
\label{GP.131}\end{multline}
where $\Omega _{L}$ is the space of fibre b-densities for
$\Pi_L:M[2]\longrightarrow M[1],$ that is the densities for the bundle of
null spaces of $\Pi_L.$ The small space(s) of pseudodifferential kernels
are similarly the space(s) of conormal densities with singularities on $D$
\begin{multline}
\Psi^{m}_{\sigma}(M[*];V)=\{B\in I^{m+\oqu(m-\kappa)}_{\sigma}(M[2],D;\Pi_L^*V\otimes\Pi_R^*V'\otimes
\Omega _{L});\\
A\equiv0\Mat H\in\cM_1(M[2])\Mif H\cap D=\emptyset\}.
\label{GP.168}\end{multline}
Here $\sigma$ refers to the class of conormal distributions admitted, for
`classical operators' it is omitted, for the class of type $1,0$ with
bounds we write $\sigma =\infty.$
\end{definition}

\noindent In the case of a concrete generalized product written $M[k;T]$
where $T$ denotes a `structure' on $M$ we use the simpler notation 
\begin{equation}
\Psi^{m}_{T}(M;V)=\Psi^{m}(M[*;T];V).
\label{GP.256}\end{equation}

Although the main concern here is with smoothing operators, recall that the
kernels of pseudodifferential operators are particular classes of conormal
distributions the basic properties of which we briefly describe. We
interpret the conormal distributions at a point as the quantization of a
simplicial compactification of the translation group.

Let $V $ be a real finite-dimensional vector space. As a translation group
the Lie algebra is also identified with $V$ and its enveloping algebra is
naturally identified with the polynomials on $V^*$ or the distributional
densities on $V$ supported at the origin. These views are intertwined by
Fourier transform
\begin{equation}
\cF:\cSp(V;\Omega )\longrightarrow \cSp(V^*).
\label{GP.36a}\end{equation}

Let $\oV^*$ be the radial compactification of $V^*$ which is a compact
manifold with boundary (a closed ball), let $\rho$ be a boundary defining
function, for instance $\rho =(1+|w|^2)^{-\ha}$ in terms of a metric on
$V^*.$ Then the `space of classical symbols of order $m$' is by definition
\begin{equation}
\rho ^{-m}\CI(\oV^*)\subset\cSp(V^*),\ m\in\bbC.
\label{GP.37}\end{equation}
Since these are tempered distibution on $V^*$ the  distributional inverse
Fourier transforms may be applied and
\begin{equation}
\cF^{-1}(\rho ^{-m}\CI(\oV^*))=I^{m+\oqu\dim V}_{\cS}(V,\{0\};\Omega )\subset\cSp(V;\Omega )
\label{GP.38}\end{equation}
is identified as the space of the classical conormal distributions of order
$m+\oqu\dim V.$ These distributions are singular only at the origin and decay rapidly with
all derivatives at infinity. The normalization of the order (due to
H\"ormander) is more natural than it might appear to be!

These space form a filtered convolution algebra
\begin{equation}
I^{m}_{\cS}(V,\{0\};\Omega )*I^{m'}_{\cS}(V,\{0\};\Omega )=I^{m+m'}_{\cS}(V,\{0\};\Omega ).
\label{GP.39}\end{equation}
The distributions supported at the origin, the derivatives of the Dirac
delta, are included in these spaces with a derivative of order $\alpha$
lying in $I^{|\alpha|+\oqu\dim V}_{\cS}(V,\{0\};\Omega),$ so the
constant-coefficent pseudodifferential operators satisfy 
\begin{equation}
\Psi^m_{\operatorname{ti}}(V)=I^{m-\oqu\dim
  V}_{\cS}(V,\{0\};\Omega)\supset\Dff^m_{\operatorname{ti}}(V)\text{ act  on }\cS(V).
\label{GP.43}\end{equation}

Even though \eqref{GP.39} follows immediately from the obvious
multiplicative property 
\begin{equation}
\rho ^{-m}\CI(\oV^*)\cdot \rho ^{-m'}\CI(\oV^*)=\rho ^{-m-m'}\CI(\oV^*)
\label{GP.41}\end{equation}
it is convenient to realize it in terms of functorial properties
of conormal functions and the mapping diagram 
\begin{equation}
\xymatrix{
&V\\
&V^2\ar[u]^{\pi_C}\ar[dl]_{\pi_S}\ar[dr]^{\pi_F}\\
V&&V
}\quad
\xymatrix{
&v\\
&(v,w)\ar@{|->}[u]\ar@{|->}[dl]\ar@{|->}[dr]\\
v-w&&w.
}
\label{GP.40}\end{equation}
Namely 
\begin{equation}
\kappa _1*\kappa _2=(\pi_C)_*(\pi_S^*\kappa _1\cdot \pi_F^*\kappa _2).
\label{GP.42}\end{equation}

In this case the three (face) maps in
\eqref{GP.40} are projections with fibre and base $V.$ Under the lower two
maps the inverse images of zero in $V$ are respectively 
\begin{equation}
\pi_S^{-1}(\{0\})= \{(v,v)\in V^2;v\in
V\}=\Diag(V^2),\ \pi_F^{-1}(\{0\})=V\times \{0\}.
\label{GP.44}\end{equation}
In general distributions cannot be pulled back under smooth maps but in
this case this is admissible by wavefrontset considerations and the
$I^{m+\oqu\dim V}(V,\{0\};\Omega)$ pull back into the more general spaces
$I^m(V^2;\pi_{S,F}^*(\{0\});\pi_{S,F}^*\Omega)$ dicussed below for an
embedded submanifold of a manifold. The transversality of the two manifolds
$\pi_{S}^*(\{0\})$ and $\pi_{F}^*(\{0\})$ means that the product on the
right in \eqref{GP.42} is well-defined. It is a `conormal distribution of
product type'; the tensor product of the two pulled-back density lines is
canonically the density line of $V^2.$ The `composite' projection $\pi_C$
is transversal to the two lifts of $\{0\}$ and as a result the push-forward
is smooth outside the image $\pi_C(\pi_{S}^*(\{0\})\cap
\pi_{F}^*(\{0\}))=\{0\}\in V$ and \eqref{GP.42} follows from an appropriate
push-forward theorem for conormal distributions.

So, although the discussion of \eqref{GP.42} is far more involved than the
use of the Fourier transform to prove \eqref{GP.39} which becomes 
\begin{equation}
\Psi^m_{\operatorname{ti}}(V)\circ
\Psi^{m'}_{\operatorname{ti}}(V)=\Psi^{m+m'}_{\operatorname{ti}}(V)
\label{GP.45}\end{equation}
it is this latter approach which is generalized below.

The main interest in the space $I^M_\infty$ and corresponding kernels is
the fact that they include the classical space $I^M$ and \ci\ is dense in
$I^M_\infty$ in the topology of $I^{M'}_\infty$ for any $M'>M.$ 

\begin{theorem}\label{GP.169} The small pseudodifferential kernels for a
  compact connected generalized product define an order-filtered
  associative algebra with identity with a surjective and multiplicative
  symbol map
\begin{equation}
\sigma:\Psi^{m}(M[*];V)\longrightarrow S^{m}(E^*M[1];\hom(V))
\label{GP.170}\end{equation}
which gives a short exact sequence for each $m$  
\begin{equation}
\xymatrix{
\Psi^{m-1}(M[*];V)\ar@{^(->}[r]&\Psi^{m}(M[*];V)\ar@{>>}[r] & S^{m}(E^*M[1];\hom(V)).
}
\label{GP.171}\end{equation}
Here $EM[1]$ is the structure bundle of the associated Lie groupoid.
\end{theorem}

\begin{proof} We concentrate on the smoothing operators, since the
extension to pseudodifferential operators only involves the behaviour of
conormal singularities at the diagonals and this turns out to be
essentially the same as in the case of a fibration.

The product is defined by direct generalization of \eqref{GP.157}. Thus if
$A$ and $B\in\Psi^{-\infty}(M[*];V)$ then the kernel of the composite is 
\begin{equation}
K_{AB}=(\Pi_C)_*\left(\Pi_S^*A\cdot\Pi_{1}^*\nu\times\Pi_F^*B\right)/
\Pi_{1}^*\nu
\label{GP.310}\end{equation}

To see that this is well-defined first note that the densities behave in the
expected manner: 
\begin{equation}
(\Pi_L\Pi_S)^*\Omega _{\bo}(M[1])\otimes
\Pi_S^*\Omega (\nul((\Pi_L)_*)\otimes \Pi_F^*\Omega (\nul((\Pi_L)_*)\longrightarrow\Omega _{\bo}(M[3])
\label{GP.312}\end{equation}
is a smooth map of line bundles which is non-vanishing near the triple
diagonal. The product of the kernels has simple support relative to $\Pi_C$
so \eqref{GP.310} defines a bilinear map 
\begin{equation}
\Psi^{-\infty}(M[*];V)\times \Psi^{-\infty}(M[*];V)\longrightarrow \Psi^{-\infty}(M[*];V).
\label{GP.313}\end{equation}

The space $M[4]$ can be used to see that this product is
associative. Consider the structural b-fibrations  
\begin{equation}
\Pi_{[i,j]}:M[4]\longrightarrow M[2],\ 1\le i< j\le 4
\label{GP.314}\end{equation}
associated to the strictly increasing map $I_{i,j}:J(2)\longrightarrow J
(4)$ which omits $i$ and $j.$ The densities behave as in \eqref{GP.312} so giving a trilinear map  
\begin{equation}
\begin{gathered}
\Psi^{-\infty}(M[*];V)\times \Psi^{-\infty}(M[*];V)\times \Psi^{-\infty}(M[*];V)\longrightarrow \Psi^{-\infty}(M[*];V),\\
(A,B,C)\longmapsto (\Pi_{2,3})_*\left(\Pi_{3,4}^*A\cdot\Pi_{L}^*\nu\times\Pi_{1,4}^*B\times \Pi_{1,2}\right)/
\Pi_{1}^*\nu.
\end{gathered}
\label{GP.315}\end{equation}
This factors through the products on the left and right to give associativity.

The pseudodifferential extension follows from the structure of the
diagonals, especially \eqref{GP.221}.
\end{proof}
\begin{proposition}\label{GP.132} For a stretched product, $M[k;T],$ the
kernels in \eqref{GP.131} of the (small) algebra of smoothing operators,
act as linear operators on $\CI([M[1];V)$ by analogy with (and extension of)
  \eqref{2.7.2023.6}:
\begin{equation}
Au=
(\Pi_L^T)_*\left(A\cdot(\Pi_2^T)^*u\cdot(\Pi_L)^*\nu_{\bo}\right)/\nu_{\bo},\
0<\nu_{\bo}\in\CI(M[1;T];\Omega_{\bo}).
\label{GP.133}\end{equation}
\end{proposition}

\begin{proof} As shown above the diagonal boundary hypersurfaces of $M[2],$
  those meeting $D,$ map to the corresponding hypersurfaces of $M[1]$ under
  both projections. Thus the kernels in \eqref{GP.131} are simply supported
  relative to $\Pi_L.$ The same is therefore true for the argument of
  the push-forward in \eqref{GP.133} which by \eqref{GP.124} is a b-density
  on $M[2;T].$ With the independence of choice of $\nu _{\bo}$ following as
  before that this gives the bilinear map
\begin{equation}
\Psi^{-\infty}_T(M;V)\times \CI(M[1];V)\longrightarrow \CI[M[1];V).
\label{GP.134}\end{equation}
Thus it only remains to prove that composites lie in the same space, that 
\begin{equation}
\Psi^{-\infty}_T(M[*];V)\circ\Psi^{-\infty}_T(M[*];V)\subset\Psi^{-\infty}_T(M[*];V) .
\label{SCL.1}\end{equation}

To do so we use properties of the triple product and the three
stretched projections 
\begin{equation}
\xymatrix{
&M[2]&\\
&M[3]\ar[u]_{\Pi_C}\ar[dl]_{\Pi_S}\ar[dr]^{\Pi_F}\\
M[2]&&M[2].
}
\label{SCL.2}\end{equation}

The kernel of the product of two elements, $A_1$ and $A_2$ is given as the
push-forward
\begin{equation}
(\Pi_C^{T})_*\left((\Pi_S^{T})^*(A_1)\cdot(\Pi_F^{T})^*(A_2)\cdot(\Pi_C^{T})^*(\nu)\right)/\nu
\label{SCL.5}\end{equation}

Similarly the first factor vanishes to infinite order at the front faces
other than those corresponding to $T$ and $D_{1,2}.$ Thus the product
vanishes to infinite order except at the triple front face correponding to
$T(\phi).$ Thus
before push-forward the numerator in \eqref{SCL.5} is a smooth
b-density. Then \eqref{3.7.2023.1} applies to show that the push-foward is
a smooth b-density on $M[2;T]$ to get the product \eqref{SCL.1}.

The multiplicative property (and extension to bundle coefficients) 
\begin{equation}
\sigma _{T}(A_1A_2)=\sigma _{T}(A_1)\cdot\sigma _{T}(A_2)
\label{SCL.10}\end{equation}
then follows either by use of an appropriate oscillatory testing argument
or by identifying the product of the kernels, restricted to the front face
in \eqref{SCL.5} with the kernel of the fibrewise convolution of the
restricted kernels.
\end{proof}

\section{Semiclassical generalized product}\label{SGP}

The construction of the Atiyah-Singer map in \cite{IndBunGer} uses the
semiclassical, and more generally the adiabatic, calculus associated to a
fibration or iterated fibration. This is discussed explicitly in the next
section, but Theorem~\ref{GP.217} allows a semiclassical generalized
product to be constructed from a given generalized product, which as usual
we assume to be compact and connected.

Let $M[k]$ be such a generalized product. Then for each $k$ the products 
\begin{equation}
[0,1]\times M[k]
\label{GP.223}\end{equation}
also form a generalized product with all structure maps being extended by
the identity on $[0,1].$ Then Theorem~\ref{GP.217} shows that the
(multi)diagonals $\cD\subset M[k]$ form, for each $k,$ a p-clean collection
of submanifolds closed under intersection.

It follows that the collections 
\begin{equation}
\cD_{\scl}=\{\{0\}\times D;D\in\cD\}\subset [0,1]\times M[k]
\label{GP.224}\end{equation}
also form a p-clean collection closed under intersection.

\begin{definition}\label{GP.225} The semiclassical spaces corresponding to
  a compact, connected generalized product are the compact manifolds with
  corners 
\begin{equation}
M[k;\scl]=[[0,1]\times M[k];\cD_{\scl}]
\label{GP.226}\end{equation}
where as usual the blow-up is to be performed in a size order using
Proposition~\ref{GP.82}.
\end{definition}

\noindent When a particular generalized product is denoted $M[k;T]$ with
`$T$' being some label to indicate the structure then the semiclassical
spaces can be denoted $M[k;\scl T].$

\begin{theorem}\label{GP.227} The semiclassical spaces associated to a
  compact connected generalized product $M[k]$ form a compact connected
  generalized product.
\end{theorem}

\begin{proof} The collection $\cD$ is invariant under the structural
  diffeomorphisms of $M[k]$ so $\cD_{\scl}$ is invariant under the
  structural diffeomorphism of $[0,1]\times M[k]$ and it follows that these
  lift to diffeomorphisms of $M[k;\scl].$ Thus Lemma~\ref{GP.81} shows that
  it suffices to check that the b-fibrations and p-embeddings
\begin{equation}
\begin{gathered}
\Pi_k:(0,1]\times M[k]\longrightarrow (0,1]\times M[k-1],\\
D_{k-1,k}:(0,1]\times M[k-1]\longrightarrow (0,1]\times M[k],\ k>1,
\end{gathered}
\label{GP.228}\end{equation}
extend to b-fibrations and p-embeddings of the $M[*;\scl].$ 

The decomposition \eqref{GP.349} of the multidiagonals in
$\cD(k)$ induces a corresponding decomposition $\cD_{\scl}$
\begin{equation}
\cD_{\scl}(k)=\cG\sqcup\cR
\label{GP.229}\end{equation}
where $\cG$ consists of the centres `which do not involve $k$',
i.e.\ the $D_{\FP}\times\{0\}$ where $k$ is a singleton and the remaining
elements, in $\cR,$ `do involve $k$' so are of the form
$D_{\FQ}\times\{0\}$ where $k\in\cQ_j$ for some $j.$
Clearly these two collections satisfy the hypotheses of the Corollary to
Theorem~\ref{GP.323} so
\begin{equation}
M[k;\scl]=[[[0,1]\times M[k];\cG];\cR].
\label{GP.230}\end{equation}

From Lemma~~\ref{GP.350} it follows under that the product extension of $\Pi_k$
to $[0,1]\times M[k]$ each $\{0\}\times D_{\FP}$ is the preimage of
$\{0\}\times D_{\FP'}$ corresponding to removing the singleton $k.$ Thus the
under successive blow-ups of the elements $\{0\}\times D_{\FP}$ in $\cG$ and
the corresponding elements $\{0\}\times D_{\FP'}$ forming $\cD_{\scl}(k-1),$
Proposition~\ref{GP.352} applies and shows that $\Pi_k$ lifts to a
b-fibration from $[[0,1]\times M[k];\cG]$ to $M[k;\scl].$ After the
blow-ups of the elements of $\cG$ each element $\{0\}\times D_{\FQ}\in\cR$
is mapped by the lift of $\Pi_k$ onto the boundary hypersurface of
$M[k-1,\scl]$ which is the front face for the blow-up of
$\{0\}\times D_{\FQ'}.$ Thus Proposition~\ref{GP.352} applies to each
blow-up showing that $\Pi_k$ lifts to a b-fibration of $M[k,\scl]$ to
$M[k-1,\scl].$ 

So consider the p-embedding in \eqref{GP.228}. Appending $[0,1]\times
D_{k-1,k}$ to $\cD_{\scl}(k)$ gives a p-clean collection closed under
intersection with $[0,1]\times D_{k-1,k}$ as the largest element, so it
certainly lifts to a p-submanifold of $M[k;\scl].$ To see that it is a
p-embedding of $M[k-1,\scl]$ consider the decomposition the collection of
the centres as
\begin{equation}
\cD_{\scl}(k)=\cF_{k-1,k}\sqcup\cG\sqcup\cF.
\label{GP.261}\end{equation}
Here the diagonals defining $\cF_{k-1,k}$ have $k-1$ and $k$ appearing in
the same set in the partition. The elements of $\cF$ are those in which
both $k-1$ and $k$ are singletons and $\cG$ is the remainder, corresponding
to the partitions where $k-1$ and $k$ appear in different sets of the
partition, but at most one is a singleton. Clearly $\cF_{k-1,k}$ and $\cF$
are closed under intersection and the union of the first two sets in
\eqref{GP.261} is also closed under intersection. So a size order on each
gives an intersection order on the union which is permissible and 
\begin{equation*}
M[k,\scl]=[[[[0,1]\times M[k];\cF_{k-1,k}];\cG];\cF].
\label{GP.355}\end{equation*}

Under the extension of the reflection $R_{k-1,k}$ the elements of
$\cF_{k-1,k}$ and $\cF,$ and $[0,1]\times D_{k-1,k},$ are invariant whereas
each element of $\cG$ is mapped to a different element of $\cG.$ It follows
that under the blow-up of $\cF_{k-1,k}$ in \eqref{GP.355} $R_{k-1,k}$ lifts
to be a diffeomorphism leaving the lift of $[0,1]\times D_{k-1,k}$
invariant and exchanging the lifts of the elements of $\cG.$ Now the
intersection an an elment of $\cG$ and its image under $R_{k-1,k}$ lies in
$\cF_{k-1,k},$ so after this first partial resolution elements of $\cG$ are
disjoint from their images under $R_{k-1,k}$ and hence are disjoint from
the lift of $[0,1]\times D_{k-1,k}.$ Thus in a neighbourhood of the lift of
$[0,1]\times D_{k-1,k}$ the blow-up of the elements of $\cG$ are
irrelevant. Now as the image of the embedding of $[0,1]\times M[k-1],$ each
element of $\cD_{\scl}(k-1)$ is the intersection of a unique element of
either $\cF_{k-1,k}$ or $\cF$ with $[0,1]\times D_{k-1,k}.$ Thus, applying
Lemma~\ref{GP.356} repeatedly shows that the p-embedding in \eqref{GP.228}
lifts to a p-embedding.

Noting that connectedness is preserved throughout, $M[k;\scl]$ is a
connected generalized product.
\end{proof}

\begin{remark}\label{GP.274} If a group $K$ acts as diffeomorphisms on the
  $M[k]$ with the actions intertwined by the strucural maps then it lifts
  to such an action on $M[k;\scl].$
\end{remark}

Next we introduce a simplicial compactification of the translation group
which we identify with a vector space $V,$ its Lie algebra. It is very
closely related to `many-body' spaces in scattering theory. See in
particular the survey by Vasy \cite{MR3098645}. In $V^{k-1}$ consider the
`diagonal' subspaces defined in terms of the same variable denoted $v_j$ in
the $j$th factor
\begin{equation}
D_{j}=\{v_j=0\},\ D_{jl}=\{v_j=v_l\},\ j,\ l=1,\dots,k-1,\ j<l.
\label{GP.263}\end{equation}
These are the images of $V^{k-2}$ under the obvious inclusion maps.

\begin{proposition}\label{GP.264} The p-submanifolds formed from the
intersections of subcollections of the $D_*$ with the sphere forming the
boundary of the radial compactification of $V^{k-1},$ form a p-clean
collection $\cI,$ closed under intersection, and
\begin{equation}
\overline{E}[k]=[\overline{V^{k-1}};\cI]
\label{GP.265}\end{equation}
is a simplicial compactification of the translation group $E.$
\end{proposition}

\begin{proof} Consider the semiclassical product constructed from the $V^k$
  as a `generalized' product. The first centre blown up in the construction
  of the space at level $k$ is the $k$-fold diagonal at $\epsilon =0.$ It
  is naturally $V\times \overline{V^{k-1}}$ with the boundary of the
  $k$-fold diagonal being $V\times\{0\}.$ Subsequent blow-ups are
  transversal to the boundary and meet it at precisely the $\cI$ in
  \eqref{GP.265}. It follows that these manifolds 
\begin{equation}
V\times \overline{E}[k]
\label{GP.268}\end{equation}
form a generalized product with all structure maps being invariant under the
translation action through the action on the first factor. Thus
\eqref{GP.265} is indeed a generalized product and satisfies the conditions
of Definition~\ref{GP.144} for the translation group.
\end{proof}

Note that the product action of the group of linear transformation of $L$
on each $V^{k-1}$ lifts to an action on $\overline{E}[k].$

\begin{proposition}\label{GP.266} The boundary generalized product
  associated to $H_{\scl}=\{0\}\times M[1]$ of $M[1,\scl],$ of a semiclassical
    generalized product, fibres over $M[1]$ 
\begin{equation}
\xymatrix{
\overline{E}[k]\ar@{^(->}[r]&H_{\scl}[k]\ar[d]\\
&M[1]
}
\label{GP.267}\end{equation}
with fibres at each point forming the simplicial compactification of the
fibre of the Lie algebroid $EM[1]\longrightarrow M[1]$ as a vector space.
\end{proposition}

\begin{proof} This follows by repeating the argument used above to prove
  Proposition~\ref{GP.264}. 
\end{proof}

\section{Adiabatic stretched product}\label{ASP}

The primary use of a generalized product in \cite{IndBunGer} is for the
construction of the associated algebra of smoothing operators for the
semiclassical product associated by Theorem~\ref{GP.227} to the fibre
products of a fibre bundle of compact manifolds without boundary. In fact
to allow `operators with values in smoothing operators' we proceed to
discuss the more general adiabatic case, which can be thought of as
`partially semiclassical'. For a single manifold $M$ this calculus was
introduced in work with Mazzeo, \cite{MR90m:58004}, corresponding to a
fibration of a compact manifold.

Since we wish to include smooth parameters, consider an iterated fibration 
\begin{equation}
\xymatrix{
Z\ar@{-}[r]&M\ar[d]_{\gamma}\ar@/^11pt/[dd]^{\phi}\\
Q\ar@{-}[r]&F\ar[d]_{\psi}\\
&Y
}
\label{GP.89}\end{equation}
Here we work, at least initially, in the context of connected compact
manifolds without boundary so the maps $\gamma$ and $\psi$ are simply
submersions and $\phi =\psi \circ\gamma.$ The adiabatic algebra corresponds
to degeneration in a parameter which we introduce as for the semiclassical
case by taking the product with a closed interval $[0,1]_\epsilon .$
Then, as for the semiclassical product, consider the fibre products 
\begin{equation}%
[0,1]\times M^{[k]}_{\phi}.
\label{GP.91}\end{equation}

The adiabatic Lie algebroid is the \ci\ module over this space generated
by smooth vector fields on $M$ 
\begin{equation}
V+\epsilon W
\label{GP.282}\end{equation}
where $W$ is tangent to the fibres of $\psi$ and $V$ is tangent to the
fibres of the finer fibration $\phi.$

There are `diagonal' submanifolds of $M^{[k]}_{\phi}$ but with respect to
the map $\gamma.$ Thus $\gamma ^k:M^k\longrightarrow F^k$ restricts to a
fibration which we denote
\begin{equation}
\gamma ^{[k]}:M^{[k]}_{\phi}\longrightarrow F^{[k]}_{\psi}.
\label{GP.269}\end{equation}

If $\mathfrak{P}$ is a partition of $J(k)$ as discussed in \S\ref{MD} then
$D_{\mathfrak{P}}F\subset F^{[k]}_{\psi}$ and we consider the p-clean
collection of submanifolds 
\begin{equation}
\cD_{\ad}=\left\{\{0\}\times(\gamma ^{[k]})^{-1}(D_{\mathfrak{P}}F)\right\}.
\label{GP.270}\end{equation}

\begin{definition}\label{GP.98} The adiabatic manifolds associated
  to an iterated fibration \eqref{GP.89} are the
  spaces, defined by blow-up in size order using
  Proposition~\ref{GP.82},
\begin{equation}
M[k;\ad]=[[0,1]\times M^{[k]}_{\phi};\cD_{\ad}].
\label{GP.99}\end{equation}
\end{definition}

\begin{theorem}\label{GP.101} The adiabatic spaces, $M[k;\ad],$ for any
  iterated fibration of connected compact manifolds \eqref{GP.89}, form a
  regular stretched product with interior $(0,1)\times M^{[k]}_{\phi}.$
\end{theorem}

\begin{proof} The proof of Theorem~\ref{GP.227} can be followed almost
  verbatim. The factor exchange diffeomorphisms of $M^{[k]}_{\phi},$
  extended to $[0,1]\times M^{[k]}_{\phi}$ as the identity on the parameter
  space, interchange the elements of $\cD_{\ad}$ within a given
  dimension. Thus they lift to give the symmetry diffeomorphims of
  $M[k;\ad]$ and Lemma~\ref{GP.81} applies. The projections $\Pi_k$ and
  diagonal embeddings $D_{k-1,k}$ are defined on the unresolved products
  $[0,1]\times M^{[*]}_{\phi}.$ Thus it remains to show that they extend
  from the interior to b-fibrations and p-embeddings of the adiabatic
  resolution. For the lift of $\Pi_k$ \eqref{GP.229} applies but now for
  the fibre diagonals. Similarly that $D_{k-1,k}$ lifts to a p-embedding
  can be seen using the argument starting at \eqref{GP.261} again applied
  to the fibre diagonals which have the same intersection properties.
\end{proof}

It is now straightforward to check that the boundary products for the
adiabatic space associated to \eqref{GP.89} are bundles over $F$ with fibre
at each point the product $E[k]\times M^{[k]}_{\gamma}$ where $E[k]$ is the
simplicial compactification of the fibre of the $T_{\psi}F.$ 

\section{The b-stretched product}\label{BSP}

Next we consider perhaps the most basic stretched product corresponding to
a compact manifold with corners. Since it involves little more effort
consider a fibration again denoted $\phi:M\longrightarrow Y$ but where now
the fibres are compact manifolds with corners. Alternatively the bundle
case can be deduced from the diffeomorphism invariance of the construction
for a single manifold.

As already noted above, the fibre diagonal in $M^{[2]}_{\phi}$ is not a
p-submanifold if the boundary is non-trivial. To resolve this we consider
the `fixed' boundary hypersurface of $M,$ those that are mapped to $Y$ by
$\phi.$ For each such boundary hypersurface $H$ consider the corresponding
boundary faces of $M^{[k]}_{\phi}$ given by the fibre products with
factors either $M$ or $H.$ Thus for each $k$ and each $P\subset J(k)$
consider
\begin{equation}
H_{P,k}=\{(m_1,m_2,\dots,m_k)\in M^{[k]}_{\phi};m_j\in H\Mif
j\in\cP\}.
\label{GP.172}\end{equation}

\begin{lemma}\label{GP.173} The boundary faces
$H_{P,k}\in\cM_{\#(P)}(M^{[k]}_{\phi}),$ for fixed $k,$ all
$P$ and all $\phi$-fixed $H\in\cM_1(M)$ form a p-clean family,
  $\FC_{\bof}(k),$ closed under non-transversal intersection.
\end{lemma}

\begin{proof} Elements $H_{P,k}$ and $H'_{P',k}$ for
  distinct hypersurfaces and any subsets $P$ and $P'$
  intersect transversally since $H$ and $H'$ do so. For a fixed $H$ the
  collection is closed under intersection with 
\begin{equation}
H_{P,k}\cap H_{Q,k}=H_{P\cup Q,k}.
\label{GP.187}\end{equation}
\end{proof}

\begin{proposition}\label{GP.174} The manifolds  
\begin{equation}
M[k;\bof]=[M^{[k]}_\phi;\FC_{\bof}(k)]
\label{GP.175}\end{equation}
are well-defined by blowing up all centres in a size order and form a
regular stretched product.
\end{proposition}

The arguments are similar to those in Theorem~\ref{GP.227} in this
combinatorially simpler setting where partitions of $J(k)$ are replaced by
subsets but complicated by the fact that diagonals $D_{k-1,k}$ are not initially
p-submanifolds.

\begin{proof} At each level, in terms of dimension, of blow up in a
  size order the intersection of non-transversal elements has already
  been blown up, so all elements in \eqref{GP.175} are transversal (or
  disjoint) and hence can be reordered freely. It follows that $M[k;\bof]$
  is symmetric, since exchanging factors amounts to reordering with
  $\#(P)$ fixed. Thus Lemma~\ref{GP.81} applies and to see that
  this is a stretched product we only need consider the existence of the
  stretched projection $\Pi_k$ from $M[k;\bof]$ to $M[k-1;\bof],$ for each
  $k>1,$ corresponding to omission of the last factor and the existence of
  the p-embedding of $M[k-1;\bof]$ in $M[k;\bof]$ as the diagonal in the
  last two factors.

The transversality of any two centres involving different boundary
hypersurfaces means, using Lemma~\ref{GP.83}, that we may freely change
the order of blow-up between different $H.$ This effectively reduces the
discussion to the case of a manifold with boundary, proceeding over some
ordering of the (fixed) boundary hypersurfaces of $M.$

We may divide the p-clean collection as
\begin{equation}
\FC_{\bof}(k)=\cG\sqcup \cR
\label{GP.176}\end{equation}
where the first set corresponds to the $P$ which do not contain $k$
and the second to those which do. So each of these subsets is labeled by
the $Q\subset J(k-1)$ either by inclusion or by addition
of $k.$ Thus
\begin{equation}
\cG\ni C'=H_{Q,k-1}\times _\phi M,\
\cR\ni C''=H_{Q,k-1}\times _\phi H
\label{GP.177}\end{equation}
for the fixed hypersurface depending on the set. Clearly, from
\eqref{GP.187}, $\cG$ and $\cR$ are separately closed under
non-transversal intersection and the non-transversal intersection of an
element of $\cG$ with one in $\cR$ lies in $\cR$ and of smaller dimension
than either. Thus Proposition~\ref{GP.187} applies separately for each $H$
and
\begin{equation}
M[k;\bof]=[[M^{[k]}_\phi;\cG];\cR].
\label{GP.283}\end{equation}

Since they do not involve the $k$th factor the iterated blow up of $\cG$ maps
through the fibre product
\begin{equation*}
[M^{[k]}_\phi;\cG]=M[k-1;\bof]\times _\phi M.
\label{GP.284}\end{equation*}
Moreover Proposition~\ref{GP.352} applies to $\Pi_k$ at each step so it lifts to
the b-fibration 
\begin{equation}
\Pi_k:[M^{[k]}_\phi;\cG]\longrightarrow M[k-1;\bof].
\label{GP.285}\end{equation}

Proposition~\ref{GP.117} applies to all the successive blow-ups of the
centres in $\cR$ so $\Pi_k,$ and hence all the projections
$M[k,\bof]\longrightarrow M[k-1,\bof]$ exist and are b-fibrations.

To see the existence of the p-embedding $D_{k-1,k}:M[k-1,\bof]\longrightarrow
M[k,\bof]$ consider the collection $\cB_{k-1,k}$ of those boundary faces of
the form
\begin{equation}
B\times_\phi H\times _\phi H
\label{GP.287}\end{equation}
with $B$ corresponding to the same $H.$ Similarly let $\cG$ be the
collection of the form \eqref{GP.287} where one of the last two fibre
factors is $M$ and the other is $H$ and finally let $\cB$ be the collection
where the last two fibre factors are $M.$ Thus
\begin{equation}
\FC_{\bo}(k)=\cB_{k-1,k}\sqcup\cG\sqcup\cB.
\label{GP.288}\end{equation}
A size order on each collection in order is an intersection order so 
Proposition~\ref{GP.182} applies.

Moreover we can change the order on $\cB_{k-1,k}$ by placing the
$M^{[k-2]}_{\phi}\times_\phi H\times_\phi H$ first and still have an
intersection order. An elementary calculation shows that after the blow up of
these (transversal) factors the diagonal $D_{k-1,k}$ becomes a
p-submanifold. The remaining elements, which the lift of $D_{k-1,k}$ form a
collection closed under non-transversal intersection and it follows that
$D_{k-1,k}$ lifts to a p-submanifold of $M[k,\bof].$

The reflection $R_{k-1,k}$ is a diffeomorphism leaving elements of the
first and last collections in \eqref{GP.288} fixed and maping each element
of the second collection to a different element. The intersection of each
element of $\cG$ with its image under the reflection lie in $\cB_{k-1,k},$
so after these faces are blown up all such pairs are disjoint. Since
$D_{k-1,k}$ is the fixed set of $R_{k-1,k}$ the blow up of the elements of
$\cG$ do not affect the lift of $D_{k-1,k}.$ As in the proof of
Theorem~\ref{GP.227} it follows that the the lift of $D_{k-1,k}$ is naturally
diffeomorphic to $M[k-1,\bof].$

Thus the$ M[k;\bof]$ form a stretched product, compactifying the fibration
restricted to the union of the interiors of the fibres of $\phi.$
\end{proof}

The b-calculus, which is the operator algebra arising from the generalized
product, is closely related to the positive real numbers as a
multiplicative group. We compactify $(0,\infty)$ using the map
\begin{equation}
G_+=(0,\infty )_t\hookrightarrow I=[-1,1]_s\Mwhere s=(t-1)/(t+1)).
\label{GP.233}\end{equation}
Then we have constructed the stretched product $I[k;\bo].$ 

\begin{proposition}\label{GP.197} The diagonal boundary hypersurfaces of
  $I[k;\bo]$ associated to $\{1\}$ as a boundary `hypersurface' of $I$ give
  a simplicial compactification, $G_+[k],$ of the multiplicative group of the
  positive real numbers and for a bundle of compact
  manifold with corners $M$ the generalized product associated to a fixed
  boundary $H$ of $M$ in the b-stretched product is diffeomorphic to
\begin{equation}
H[k;\bof]\times \overline{G_+}[k],\ G_+=(0,\infty).
\label{GP.198}\end{equation}
\end{proposition}

\noindent The freedom in the diffeomorphism to \eqref{GP.198} arises from a
choice of boundary defining function for $H.$ 

\begin{proof} 
The boundary generalized products are easily identified. Namely, if $H$ is
a fixed boundary hypersurface of the total space $M=M[1,\bof]$ then the
diagonal boundary hypersurfaces of $M[2;\bof]$ are $H\times[-1,1]_s$ where
the interval can be identified as the compactification of the
multiplicative group 
\end{proof}

Theorem~\ref{GP.227} immediately yields a semiclassical version of the
b-stretched product; the algebra of smoothing operators associated to this
is used (in a rather minor way) in \cite{IndBunGer} and the
pseudodifferential operators much more substantially (and independently of
this discussion) by Hintz in \cite{MR4515441}.

\section{Double semiclassical product}\label{DSST}

Now we return to an iterated fibration as in \eqref{GP.89} and proceed to
define the `doubly semiclassical product' involving two semiclassical
parameters corresponding to the parameter square
\begin{equation}
\Delta=[0,1]_{\epsilon }\times[0,1]_{\delta }.
\label{SCL.28}\end{equation}

The idea is to reproduce the adiabatic product for the iterated fibration
on $\epsilon=1$ with $\delta$ as parameter, the semiclassical product for
$\phi$ on $\delta =1,$ with $\epsilon$ as parameter and a pulled-back
version of the semiclassical product for $\gamma$ on $\delta =0$ with
$\epsilon $ as parameter. The single space is
\begin{equation}
M[1;\dscl]=\Delta \times M.
\label{GP.236}\end{equation}

The Lie algebroid which is resolved by the double semiclassical product
consists of the (\ci\ module generated by the) smooth vector fields on
$\Delta \times M$ 
\begin{equation}
\epsilon V+\epsilon \delta W
\label{GP.281}\end{equation}
where $V$ and $W$ are smooth vector fields on $M$ respectively tangent to
the fibres of $\phi$ and to the fibres of $\gamma.$

We first consider the semiclassical stretched product for $\phi$ in
\eqref{GP.89} with $\epsilon$ as parameter, as constructed in \S\ref{ASP}
\begin{equation}
M[k;\scl\phi]=[[0,1]\times M^{[k]}_\phi ;\{0\}\times \FC_{\phi}]
\label{GP.280}\end{equation}
where $\FC_\phi$ is the collection of diagonals with respect to the
composite fibration. The idea then is to add the parameer $\delta \in[0,1]$
and perform adiabatic resolution over the top of the semiclassical
resolution.

Recall that the products of $\gamma :M\longrightarrow F$ lift to fibrations
$\gamma ^{[k]}:M^{[k]}_\phi \longrightarrow F^{[k]}_\psi$ and the adiabatic
resolution arises from the fibre diagonals for $\gamma,$ meaning the inverse
images under these maps of the diagonals in $F^{[k]}_\psi.$

\begin{lemma}\label{GP.253} For an iterated fibration, under the blow
  up in \eqref{GP.280} the fibre diagonals for $\gamma $ lift from
  $[0,1]\times M^{[k]}_\phi$ to a p-clean collection in $M[k,\scl\phi].$
\end{lemma}

\begin{proof} For simplicity we will generally denote the
  $\psi$-diagonals $(\gamma ^{[k]})^{-1}(D^\psi_{\mathfrak{P}})$ as
  $D^\psi_{\mathfrak{P}}\in\cD_\psi.$ To see that these lift to a p-clean
  collection in $M[k;\scl\phi]$ we examine the local structure of the full
  collection
\begin{equation}
\{0\}\times \FC_{\phi}\cup [0,1]\times D^\psi_{\mathfrak{P}}\Min [0,1]\times M^{[k]}_{\phi}
\label{16.8.2023.2}\end{equation}
based on the relationships that for each partition $\mathfrak{P}$ 
\begin{equation}
D^\psi_{\mathfrak{P}}=(\gamma
^{[k]})^{-1}(D_{\mathfrak{P}}),\ D_{\mathfrak{P}}\subset [0,1]\times
F^{[k]}_\psi,\ (\gamma^{[k]})(\{0\}\times D^\phi_{\mathfrak{P}})=\{0\}\times D^\psi_{\mathfrak{P}}.
\label{16.8.2023.3}\end{equation}

Consider first a point in the minimal diagonal $\{0\}\times D^\phi_{J(k)}.$
This is a point $(m,\dots,m).$ We may therefore take local coordinates in
$F$ near $\psi (m)$ and local coordinates in the fibre so that in terms of
the differences, $u_i$ and $y_i,$ 
\begin{multline}
\{0\}\times D^\phi_{J(k)}=\{\delta =0,\ u_1=\dots=u_{k-1}=0=y_1=\dots=y_{k-1}\},\\
D^\psi_{J(k)}=\{u_1=\dots=u_{k-1}=0\}.
\label{16.8.2023.4}\end{multline}
The other diagonals through this point then take a similar form 
\begin{multline}
\{0\}\times D^\phi_{\mathfrak{P}}=\{\delta =0,\ L_{\mathfrak{P}}(u_1,\dots,u_{k-1})=0=L_{\mathfrak{P}}(y_1,\dots,y_{k-1})\},\\
D^\psi_{\mathfrak{P}}=\{L_{\mathfrak{P}}(u_1,\dots,u_{k-1})=0\}
\label{16.8.2023.5}\end{multline}
where the collection, $L_{\mathfrak{P}}$ of linear functions of $k-1$ variables only depends on $\mathfrak{P}.$

Now consider the blow up of $\{0\}\times D^\phi_{J(k)}.$ This introduces a
new boundary hypersurface. None of the other manifolds is contained in this
centre so each lifts to the closure of the complement of the centre. Near
the interior of the front face we may use the projective coordinates
$U=u/\delta,$ $Y=y/\delta,$ $\delta \ge0$ and some more lifted
functions. All the other diagonals $\{0\}\times D^\phi_{\mathfrak{P}}$ lift
into the proper transform of $H_\delta$ so do not meet the interior. The
lifts of the interior p-submanifolds are the 
\begin{equation}
D^\psi_{\mathfrak{P}}=\{L_{\mathfrak{P}}(U_1,\dots,U_{k-1})=0\}
\label{16.8.2023.6}\end{equation}
so form a p-clean family. It remains to analyse a neighbourhood of a
boundary point of the front face. Here one of the functions $u$ or $y$
dominates the others and $\delta .$ Suppose one of the components of one of
the $u_i$ dominates. We can make a linear change of coordinates, the same
in the $u_i$ and $y_i$ (so the same for all components) and apply a
symmetry so that the first component of $u_1$ denoted $v$ is dominant. Then
$v,$ $\eta =\delta /v,$ $U_i=u_i/v$ and $Y_i=y_i/v$ become coordinates
(with the first component of $U_1$ being identically one). In these
coordinates (which do not necessarily vanish at the basepoint), the lifts
satisfy the affine equations
\begin{multline}
\beta ^{\#}(D^\psi_{\mathfrak{P}})=\{L_{\mathfrak{P}}(U_1,\dots,U_{k-1})=0\},\\
\beta ^{\#}(\{0\}\times D^\phi_{\mathfrak{P}})=\{\eta =0,\ L_{\mathfrak{P}}(U_1,\dots,U_{k-1})=0=L_{\mathfrak{P}}(U_1,\dots,U_{k-1})=0\}.
\label{16.8.2023.7}\end{multline}

The first lift passes through the chosen base point if and only if the
coordinates satisfy all the affine constraints. Then, normalizing the
coordinates by subtracting the values at the basepoint they
are all given by linear equations in the normalized $U_i.$ So they form a
local p-clean collection of interior p-submanifolds. Next consider which of
the lifts of the $\beta ^{\#}(\{0\}\times D^\phi_{\mathfrak{P}})$ pass
through the current basepoint. They must project onto one of the manifolds
just discussed which does pass through the basepoint but in addition the
$L_{\mathfrak{P}}(Y_1,\dots,Y_{k-1})$ must vanish at the base point. If
this holds these lifts are given by linear equations in terms of the
normalzed $Y_i.$ So in this case we see that the lifted submanifolds
through a give point consist of some collection of the $\beta
^{\#}(D^\psi_{\mathfrak{P}})$ defined by linear constraints
in interior coordinates $U'$ (the collection necessarily closed under
intersection) and some generally smaller collection of the $\beta
^{\#}(\{0\}\times D^\phi_{\mathfrak{P}})$ given by the same linear
constraints in the $U'$ and some extra linear constraints in the $Y'$
variables.

If none of the components of the $u_i$ is dominant then one of the
components of the $y_i$ is instead and becomes a defining function for the
front face and the same conclusion holds for the defining functions.

We will proceed by induction, over a size order on the $\{0\}\times
D^\phi_{\mathfrak{P}}.$ We make the following inductive hypothesis at a
given stage of the blow-up of the $\{0\}\times D^\phi_{\mathfrak{P}}.$
Namely we assume that the $D^\psi_{\mathfrak{Q}}$ lift to a p-clean family
of interior p-submanifolds and near any point in the lift of the boundary
$\delta =0$ the $D^\psi_{\mathfrak{Q}}$ through that point are given by
linear equations in some interior coordinates $U'$ and the $\beta
^{\#}(\{0\}\times D^\phi_{\mathfrak{r}})$ which pass through that point
correspond to a subset of the $\mathfrak{Q}$ which do so and they are
defined by the same linear equations in the $U'$ and some other linear
equations in some independent interior coordinates $Y'.$ Both collections
are manifestly closed under intersection.

The discussion above is the first step but the inductive step is very
similar. Thus we are considering a point, in the lift of $H_{\delta =0},$
in the lift of a minimal (remaining) $\{0\}\times D^\phi_{\mathfrak{P}}.$
By the inductive hypothesis above the lift of $D^\psi_{\mathfrak{P}}$ must
pass through this point but there may also be smaller
$D^\phi_{\mathfrak{Q}}$ which do so. Since the collection is closed under
intersection these must correspond to $\mathfrak{Q}\supset \mathfrak{P}.$
On the other hand there can be (the lifts of) larger
$D^\psi_{\mathfrak{Q}}$ through this point and larger $\{0\}\times
D^\phi_{\mathfrak{Q}}$ through the point corresponding to a subcollection
of the $D^\psi_{\mathfrak{Q}},$ $\mathfrak{Q}\subset\mathfrak{P}.$ So, by
the inductive hypothesis we have local coordinates $U,$ $Y$ ($\eta,$
defining the lift of $H_{\delta =0})$ in which
\begin{equation}
\beta^{\#}(D^\psi_{\mathfrak{P}})=\{L_{\mathfrak{P}}(U)=0\},\
\beta^{\#}(D^\psi_{\mathfrak{P}})=\{L_{\mathfrak{P}}(U)=0,\ l_{\mathfrak{P}}(Y)=0\}
\label{GP.249}\end{equation}
for some linear functions $L$ and $l.$ By relabelling the variables $U$ and
using independence we can suppose that 
\begin{equation}
\beta^{\#}(D^\psi_{\mathfrak{P}})=\{U'=0\},\
\beta^{\#}(D^\psi_{\mathfrak{P}})=\{U'=0,\ Y'=0\}
\label{GP.250}\end{equation}
Here $U=(U',U'')$ is some division of the variables. It follows that the
smaller $D^{\psi}_{\mathfrak{Q}}$ passing through that point must be of
  the form 
\begin{equation}
\beta^{\#}(D^\psi_{\mathfrak{Q}})=\{U'=0,\ L''_{\mathfrak{Q}}(U'')=0\},
\label{GP.251}\end{equation}
by removing any dependence of the additional linear functions $L''$ on
$U'.$ The larger $\psi$ diagonals through the point, and any $\phi$
diagonal which projects onto it and passes through the point must then be of the
form 
\begin{equation}
\beta^{\#}(D^\psi_{\mathfrak{Q}})=\{L_{\mathfrak{Q}}U'=0\},\
\beta^{\#}(\{0\}\times D^\phi_{\mathfrak{Q}})=\{L_{\mathfrak{Q}}U';l_{\mathfrak{Q}}(Y')=0\}.
\label{GP.252}\end{equation}
This allows the discussion above of the initial blow-up to be followed
closely. The variables $U'$ (and $Y''$) are not involved in the blow-up at
all the inductive hypothesis is confirmed after this blow-up so holds in
general and the Lemma is proven.
\end{proof}

Thus if $\FC_\psi$ is the collection of the $\psi$-diagonals pulled back to
$M[k;\scl\phi]$ in this way then they have the same intersection properties as
before and we define the doubly semiclassical space by  
\begin{equation}
M[k;\dscl]=[[0,1]_\delta \times M[k;\scl\phi];\{0\}\times \FC_\psi]
\label{GP.243}\end{equation}
as usual in terms of a size order in each case.

\begin{theorem}\label{GP.244} The doubly-semiclassical spaces
  \eqref{GP.243} form a regular stretched product resolving $\Delta\times
  M^{[k]}_{\phi}$ over the interior, its boundary generalized products are
  $M[k;\scl\phi]$ over $\delta =1,$ $M[k;\ad]$ over $\epsilon =1$ and the
  pull-back of $M[k;\scl\gamma]$ to $\overline{T_{\psi}F}$ over $\delta =0.$
\end{theorem}

\begin{proof} Symmetry follows by noting that the permutations on
  $M[k;\scl\phi]$ map the $\psi$-fibre diagonals functorially and so lift
  under the blow-ups to diffeomorphisms.

Thus, by Lemma~\ref{GP.81}, it remains to show that the semiclassical
structural maps
\begin{equation}
\Pi_k^{\scl\phi}:[0,1]_\delta \times M[k;\scl\phi]\longrightarrow [0,1]_\delta \times M[k-1;\scl\phi]
\label{GP.245}\end{equation}
lift to a b-fibrations and the p-embeddings
\begin{equation}
D_{k-1,k}:[0,1]\times M[k-1;\scl\phi]\longrightarrow [0,1]\times M[k;\scl\phi]
\label{GP.246}\end{equation}
lift to a p-embeddings.

As before the first follows by applying Proposition~\ref{GP.182} to the
decomposition 
\begin{equation}
\FC=\cF\sqcup\cG
\label{GP.247}\end{equation}
into the $\psi$-diagonals not involving $k,$ i.e.\ in which it is a singleton, and
those in which it lies in one of the $\cP_l.$

The former are precisely the inverse images of the corresponding diagonals,
over $\delta =0,$ in $M[k-1;\ad]$ so the b-fibration in \eqref{GP.245}
lifts to a b-fibration 
\begin{equation}
F:[[0,1]\times M[k;\ad];\cF]\longrightarrow M[k-1;\dscl].
\label{GP.248}\end{equation}
Each element of $\cG$ is associated to a partition $\mathfrak{Q}\subset
J(k-1)$ by removing $k-1,$ $k$ or both from $\mathfrak{P}\subset J(k).$ The
image under $\Pi_k^{\ad}$ of the boundary face is then the corresponding
$D_{\mathfrak{Q}}\subset \{0\}\times M[k-1;\ad].$ It follows that after the
blow-up of $\cF$ these map to the corresponding front face of $[[0,1]\times
M[k-1;\ad];\cF].$

That $D_{k-1,k},$ the $\phi$-diagonal, is a p-submanifold of $M[k;\dscl]$
follows as in the adiabatic case.
\end{proof}

\providecommand{\bysame}{\leavevmode\hbox to3em{\hrulefill}\thinspace}
\providecommand{\MR}{\relax\ifhmode\unskip\space\fi MR }
\providecommand{\MRhref}[2]{%
  \href{http://www.ams.org/mathscinet-getitem?mr=#1}{#2}
}
\providecommand{\href}[2]{#2}

\end{document}